\documentclass[preprint,3p,10pt,a4paper]{elsarticle}
\usepackage[dvipsnames]{xcolor}
\usepackage[ruled, linesnumbered]{algorithm2e}

\usepackage[english]{babel}
\usepackage{bm}
\usepackage{placeins}
\usepackage{xfrac}
\usepackage{epstopdf,cmap,amssymb,amsfonts,amsmath,mathtext,enumerate,float,natbib,indentfirst,graphicx,multirow,color,setspace,subfigure,amsthm,chngcntr,url}
\usepackage{lmodern}
\usepackage{mathrsfs}
\usepackage[colorlinks=true, breaklinks=true, pdftex]{hyperref}
\usepackage{booktabs}

\biboptions{numbers,sort&compress}

\newtheorem{theorem}{Theorem}

\newtheorem{corollary}[theorem]{Corollary}

\usepackage{lineno}
\modulolinenumbers[1]

\journal{}

\begin{document}

\begin{frontmatter}



\title{Regularity analysis and verification of Coons volume mappings}


\author[label_lnu]{Yingying Yu}
\author[label_lnu]{Yashu Liu}
\author[label_lnu]{Jiaxuan Li}
\author[label_dlut]{Xin Li}
\author[label_tud]{Ye Ji\corref{cor1}}
\author[label_dlut]{Chungang Zhu}

\affiliation[label_lnu]{organization={School of Mathematics, Liaoning Normal University},
            city={Dalian},
            postcode={116029}, 
            country={China}}

\affiliation[label_dlut]{organization={School of Mathematical Sciences, Dalian University of Technology},
            city={Dalian},
            postcode={116024},
            country={China}}

\affiliation[label_tud]{organization={Delft Institute of Applied Mathematics, Delft University of Technology},
            city={Delft},
            postcode={2628 CD},
            country={the Netherlands}}

\cortext[cor1]{Corresponding author. Email: y.ji-1@tudelft.nl}

\begin{abstract}
The Coons volume provides a classical approach for constructing three-dimensional parametric mappings via boundary surface interpolation and is widely employed in volumetric mesh generation, computer-aided geometric design, and isogeometric analysis. However, due to curvature variations and continuity limitations of the boundary surfaces, the Jacobian determinant of a Coons volume may locally vanish or become negative, resulting in a non-regular mapping. This undermines mesh quality and compromises the stability of subsequent numerical computations. Ensuring the regularity of Coons volumes is therefore critical for robust parametric modeling.

This paper develops a systematic framework for analyzing and verifying the regularity of Coons volumes. We first derive a general sufficient condition applicable to arbitrary boundary parameterizations, independent of specific analytical forms. For B\'ezier-form Coons volumes, we introduce a criterion based on the B\'ezier coefficients of the Jacobian determinant, transforming the verification problem into checking the positivity of control coefficients. Furthermore, we construct a necessary condition by applying a subdivision strategy combined with the B\'ezier blossoming technique, ensuring that regularity is preserved in all subdomains.

By integrating these conditions, we design an efficient verification algorithm whose correctness and computational performance are validated through numerical experiments. We observe that the regularity of a Coons volume is closely related to the geometric similarity of its opposite boundary surfaces. Moreover, through B\'ezier extraction, the algorithm is extended to multi-patch B-spline volumes of arbitrary topology. Numerical tests show that the method completes regularity verification in milliseconds, enabling real-time application. This work contributes both theoretical and computational tools for quality assurance in volumetric parametric modeling.
\end{abstract}




\end{frontmatter}




\section{Introduction}
\label{sec1:introduction}

In the field of Computer-Aided Geometric Design (CAGD), the Coons volume is a classical and foundational technique for constructing three-dimensional parametric solids through boundary surface interpolation \cite{farin1999discrete}. It plays a vital role in geometric modeling, serving as a bridge between boundary representation (B-rep) and volumetric solid modeling \cite{shi2010filling}. Coons volumes have found wide applications in areas such as structured volume mesh generation \cite{xu2013analysis}, and Isogeometric Analysis (IGA) \cite{hughes2005isogeometric}.

With the rapid development of CAD/CAE integration and increasing demands to reduce product development cycles, the seamless coupling of design and analysis has become a prevailing trend \cite{cottrell2009isogeometric}. In this context, robust and efficient volume mesh generation is essential in engineering and scientific computing. In mechanical design, 3D meshes are indispensable in Finite Element Analysis (FEA), enabling accurate assessment of stress, strain, and deformation under loading conditions. Applications such as the structural analysis of automotive components, turbomachinery blades, aerospace assemblies, and biomedical implants rely heavily on high-quality volumetric parameterizations or meshes to ensure simulation accuracy and numerical robustness \cite{hughes2005isogeometric, zhang2018geometric, pietroni2022hex}.

To generate such meshes, Coons volumes are often used to parameterize structured hexahedral regions \cite{xu2015efficient}. The classical approach involves defining a solid bounded by six parametric surfaces and constructing the interior mapping via surface interpolation, as illustrated in Figure~\ref{fig1:boundary_surfaces_and_Coons_volume}. However, in practice, a single parametric block is often insufficient for complex domains. More critically, ensuring the regularity of the parametric mapping, i.e., maintaining a strictly positive Jacobian determinant throughout the domain, is essential for mesh quality and for the convergence of downstream solvers such as finite element, isogeometric analysis, or computational fluid dynamics methods \cite{zhao2021geometric, yu2021conditions, yu2023sufficient}.

\begin{figure}[H]
    \centering
    \includegraphics[width=0.75\linewidth]{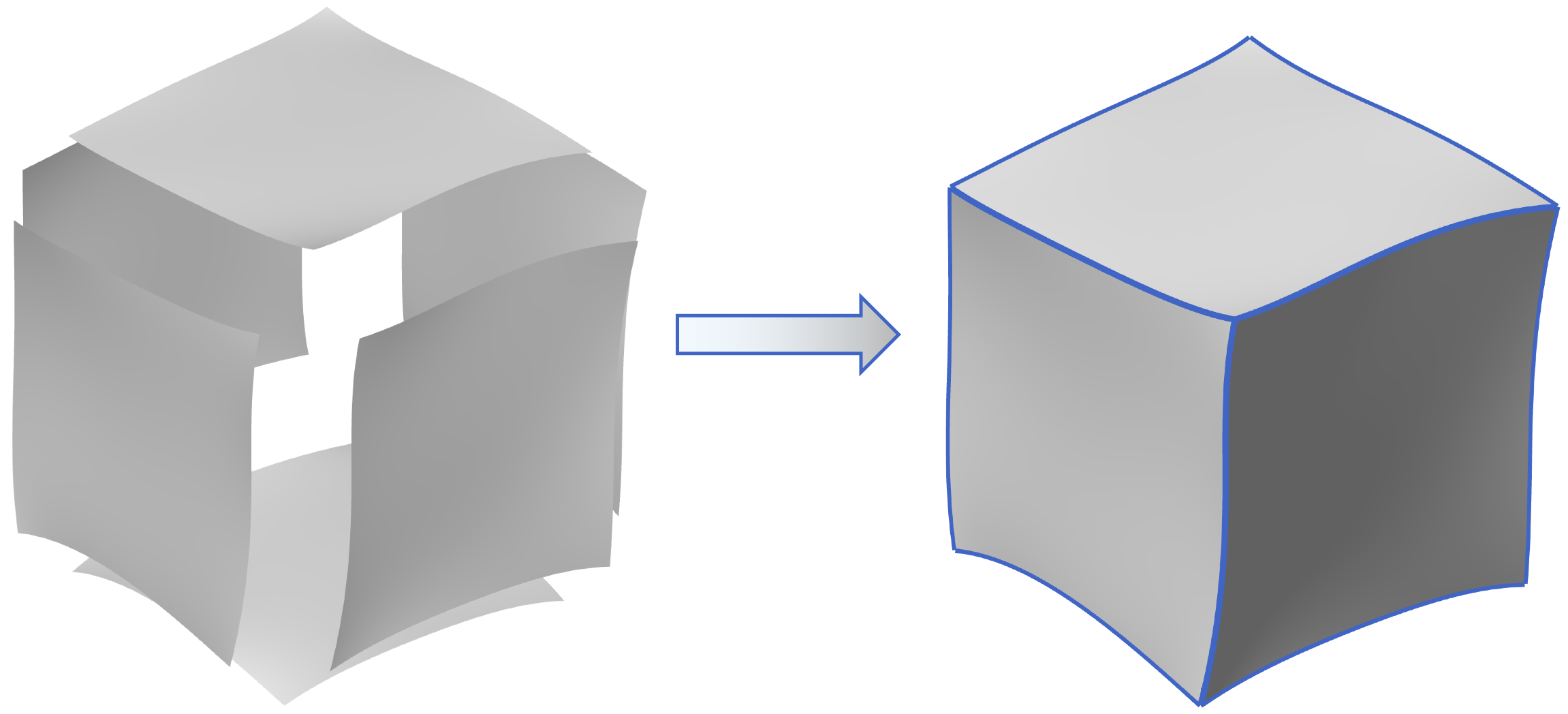}
   \caption{A volumetric domain enclosed by six boundary surfaces and constructed using Coons volume mapping.}
\label{fig1:boundary_surfaces_and_Coons_volume}
\end{figure}

However, due to geometric complexity, the boundary surfaces may exhibit significant curvature variation or lack smooth connections, especially in practical modeling scenarios. Such irregularities may lead to degeneracies or sign changes in the Jacobian determinant, violating the regularity condition. These local singularities may cause low-quality or invalid mesh elements, which in turn severely affect the accuracy and robustness of numerical simulations. Therefore, establishing rigorous regularity criteria for Coons volume mappings under general geometric conditions is of both theoretical and practical importance.

This paper addresses the above challenges by proposing new theoretical results and algorithms for regularity analysis of Coons volumes. The main contributions are as follows:
\begin{itemize}
\item \textbf{Regularity conditions for general Coons volumes (Theorem~\ref{theorem 1}).} 
We derive a sufficient condition for the regularity of Coons volumes bounded by general parametric surfaces. This condition does not rely on the specific functional form of the boundaries or blending functions, ensuring wide applicability. Since regularity is closely tied to curvature variation and similarity between opposite boundaries, our formulation offers a geometric perspective for quality control.

\item \textbf{Regularity conditions for B\'ezier-form Coons volumes (Theorem~\ref{theorem 2}).}
In many practical cases, boundary surfaces are represented in B\'ezier or spline form. For such representations, we reformulate the regularity condition as a positivity test on the B\'ezier coefficients of the Jacobian determinant. This approach improves interpretability and computational efficiency, while accommodating more flexible boundary configurations.

\item \textbf{Necessary conditions via subdivision and blossoming (Theorem~\ref{theorem 3}).}
We propose a necessary condition for regularity based on recursive subdivision and the B\'ezier blossoming algorithm. If a volume is regular, then all its subdomains after subdivision must also be regular. This condition offers a new theoretical tool for localized quality analysis.

\item \textbf{Efficient regularity verification algorithm.}
By combining the sufficient and necessary conditions above, we design an efficient algorithm for verifying regularity in Coons volumes. The algorithm demonstrates robustness across various geometric configurations and is applicable to both single-block and multi-patch volume constructions.
\end{itemize}

The remainder of this paper is organized as follows. Section~\ref{sec2:Theoretical Analysis} presents the theoretical analysis of the regularity conditions for Coons volumes, including general parametric surfaces and B\'ezier representations, along with the derivation of sufficient and necessary conditions. Section~\ref{sec3:numerical experiments} provides numerical experiments that demonstrate the effectiveness, computational efficiency, and practical robustness of the proposed algorithm across different examples. Finally, Section~\ref{sec4:conclusions} summarizes the main contributions of this work and outlines potential directions for future research.

\section{Theoretical analysis}
\label{sec2:Theoretical Analysis}

\subsection{Coons volume mapping}
\label{sec201:Coons Volume Mapping}

Let $\bm{\mathcal{S}}_1(u, v)$, $\bm{\mathcal{S}}_2(u, v)$, $\bm{\mathcal{S}}_3(u, w)$, $\bm{\mathcal{S}}_4(u, w)$, $\bm{\mathcal{S}}_5(v, w)$, and $\bm{\mathcal{S}}_6(v, w)$: $[0,1]^2 \to \mathbb{R}^3$ be six parametric surfaces defined over the unit square $[0,1]^2${\color{red}, as illustrated in Figure~\ref{fig2:notation}}. These surfaces represent the six faces of a hexahedral domain, corresponding to the front/back, left/right, and top/bottom boundaries of a three-dimensional parametric solid (cf. Figure~\ref{fig1:boundary_surfaces_and_Coons_volume}).

To ensure a well-defined volume, adjacent boundary surfaces are required to satisfy the following continuity conditions along their shared edges:
\begin{equation}
\begin{aligned}
  &\bm{\mathcal{S}}_1(0, v) = \bm{\mathcal{S}}_5(v, 0), \quad && \bm{\mathcal{S}}_1(1, v) = \bm{\mathcal{S}}_6(v, 1), \\
  &\bm{\mathcal{S}}_1(u, 0) = \bm{\mathcal{S}}_3(u, 0), \quad && \bm{\mathcal{S}}_1(u, 1) = \bm{\mathcal{S}}_4(u, 0), \\
  &\bm{\mathcal{S}}_2(u, 0) = \bm{\mathcal{S}}_3(u, 1), \quad && \bm{\mathcal{S}}_2(u, 1) = \bm{\mathcal{S}}_4(u, 1), \\
  &\bm{\mathcal{S}}_2(0, v) = \bm{\mathcal{S}}_5(v, 1), \quad && \bm{\mathcal{S}}_2(1, v) = \bm{\mathcal{S}}_6(v, 0), \\
  &\bm{\mathcal{S}}_3(0, w) = \bm{\mathcal{S}}_5(0, w), \quad && \bm{\mathcal{S}}_3(1, w) = \bm{\mathcal{S}}_6(0, w), \\
  &\bm{\mathcal{S}}_4(0, w) = \bm{\mathcal{S}}_5(1, w), \quad && \bm{\mathcal{S}}_4(1, w) = \bm{\mathcal{S}}_6(1, w).
\end{aligned}
\label{eq1:boundary surface}
\end{equation}

\begin{figure}[H]
    \centering
    \includegraphics[width=0.7\linewidth]{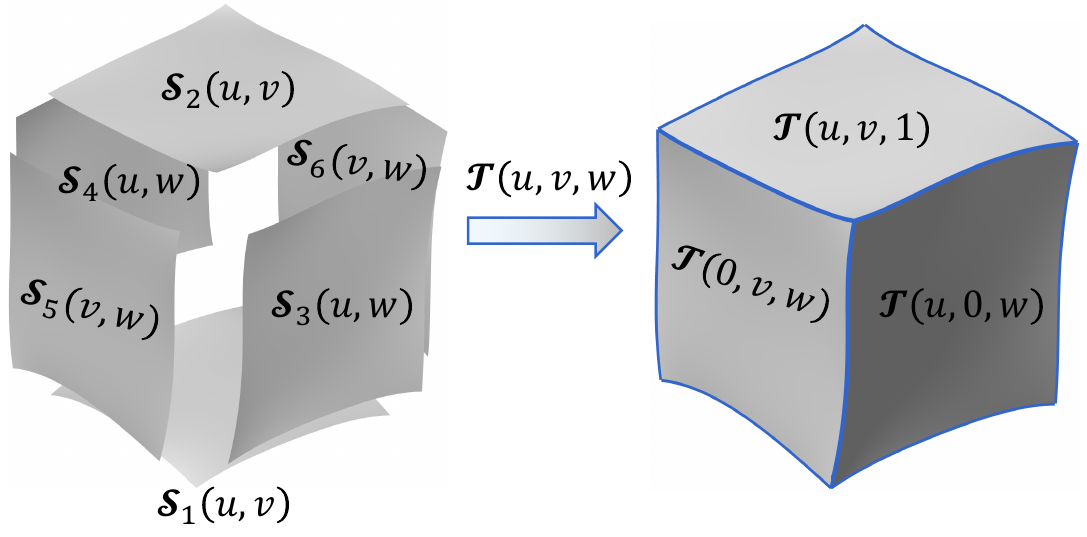}
    \caption{\color{red}Annotated notation of the Coons volume mapping.}
    \label{fig2:notation}
\end{figure}

Assume that aside from the aforementioned shared edges, the six boundary surfaces do not intersect each other. Based on Coons interpolation theory, we define a mapping $\bm{\mathcal{T}} : [0,1]^3 \to \mathbb{R}^3$ from the unit cube to a solid bounded by these six surfaces. The explicit formulation of the Coons volume mapping is given as follows~\cite{farin1999discrete,xu2015efficient}:
\begin{equation}
\begin{aligned}
     \bm{\mathcal{T}}(u,v,w) &= F_0(u)\bm{\mathcal{T}}(0,v,w)+F_1(u) \bm{\mathcal{T}}(1,v,w) + F_0(v) \bm{\mathcal{T}}(u,0,w) + F_1(v) \bm{\mathcal{T}}(u,1,w)\\
     &+F_0(w)\bm{\mathcal{T}}(u,v,0) + F_1(w) \bm{\mathcal{T}}(u,v,1)-[F_0(u), F_1(u)] \begin{bmatrix} \bm{\mathcal{T}}(0,0,w) &\bm{\mathcal{T}}(0,1,w) \\ \bm{\mathcal{T}}(1,0,w) &\bm{\mathcal{T}}(1,1,w) \end{bmatrix} \begin{bmatrix} F_0(v) \\ F_1(v) \end{bmatrix} \\
     &- [F_0(v), F_1(v))] \begin{bmatrix} \bm{\mathcal{T}}(u,0,0) &\bm{\mathcal{T}}(u,0,1) \\ \bm{\mathcal{T}}(u,1,0) &\bm{\mathcal{T}}(u,1,1) \end{bmatrix} \begin{bmatrix} F_0(w) \\ F_1(w) \end{bmatrix} \\
     &- [F_0(w), F_1(w)] \begin{bmatrix} \bm{\mathcal{T}}(0,v,0) & \bm{\mathcal{T}}(1,v,0) \\ \bm{\mathcal{T}}(0,v,1) & \bm{\mathcal{T}}(1,v,1) \end{bmatrix} \begin{bmatrix} F_0(u) \\ F_1(u) \end{bmatrix}\\
     &+ F_0(w)\begin{bmatrix} F_0(u), F_1(u)\end{bmatrix} \begin{bmatrix} \bm{\mathcal{T}}(0,0,0) & \bm{\mathcal{T}}(0,1,0) \\ \bm{\mathcal{T}}(1,0,0) & \bm{\mathcal{T}}(1,1,0) \end{bmatrix} \begin{bmatrix}F_0(v) \\ F_1(v)\end{bmatrix} \\
     &+ F_1(w) \begin{bmatrix}F_0(u), F_1(u)\end{bmatrix} \begin{bmatrix} \bm{\mathcal{T}}(0,0,1) & \bm{\mathcal{T}}(0,1,1) \\\bm{\mathcal{T}}(1,0,1) & \bm{\mathcal{T}}(1,1,1) \end{bmatrix} \begin{bmatrix} F_0(v)\\ F_1(v) \end{bmatrix},
\end{aligned}
\end{equation}
where $F_0(t)$ and $F_1(t)$ are strictly increasing blending functions that satisfy the Kronecker-delta interpolation property
\begin{equation}
    F_i(j) = \delta_{ij}, \quad i,j \in {0,1}, \quad \text{and} \quad F_0(t) + F_1(t) = 1, \quad \forall t \in [0,1].
\end{equation}
The Coons volume $\bm{\mathcal{T}}(u,v,w)$ satisfies the following boundary constraints:
\begin{equation}
\begin{aligned}
    \bm{\mathcal{S}}_1(u, v)&=\bm{\mathcal{T}}(u,v, 0), \quad \bm{\mathcal{S}}_2(u, v)=\bm{\mathcal{T}}(u,v, 1), \quad \bm{\mathcal{S}}_3(u, w)=\bm{\mathcal{T}}(u,0,w) ,  \\
    \bm{\mathcal{S}}_4(u, w)&=\bm{\mathcal{T}}(u, 1,w), \quad \bm{\mathcal{S}}_5(v,w)=\bm{\mathcal{T}}(1,v,w), \quad \bm{\mathcal{S}}_6(v,w)=\bm{\mathcal{T}}(0,v,w).
\end{aligned}
\end{equation}

If the boundary surfaces $\bm{\mathcal{S}}_i$ ($i=1,\dots,6$) are smooth and the blending functions are continuously differentiable, then the resulting Coons volume $\bm{\mathcal{T}}$ is also differentiable.

Common choices for the blending functions include:
\begin{itemize}
    \item Linear blending functions
    \begin{equation}
    \begin{aligned}
        F_0(t)=t, \quad \quad F_1(t)=1-t.
    \end{aligned}
    \end{equation}
    
    \item Cubic B\'{e}zier blending functions
        \begin{equation}
            \begin{aligned}
                F_0(t) &= 2t^3-3t^2+1 = \mathcal{B}^3_0(t)+\mathcal{B}^3_1(t),\\
                F_1(t) &= -2t^3+3t^2  = \mathcal{B}^3_2(t)+\mathcal{B}^3_3(t),\\
            \end{aligned}
        \end{equation}
    where $\mathcal{B}^3_i(t) = \binom{3}{i} t^i (1-t)^{3-i}(i=0,1,2,3)$ represents the cubic Bernstein basis functions.
\end{itemize}

The first-order partial derivatives of the Coons mapping $\bm{\mathcal{T}}$ with respect to the parametric variables $u$, $v$, and $w$ are given by
\begin{equation}
    \begin{aligned}
     \bm{\mathcal{T}}_u(u,v,w) &=  F_1'(u)\bm \lambda_1(v,w) + \bm\phi_1(u,v,w),\\
     \bm{\mathcal{T}}_v(u,v,w) &= F_1'(v) \bm \lambda_2(u,w) + \bm \phi_2(u,v,w),\\
     \bm{\mathcal{T}}_w(u,v,w)&= F_1'(w) \bm \lambda_3(u,v) + \bm \phi_3(u,v,w),
    \end{aligned}
\end{equation}
where
\begin{equation}
    \begin{aligned}
    \bm \lambda_1(v, w)&=\bm{\mathcal{S}}_5(v,w)-\bm{\mathcal{S}}_6(v,w)+F_0(v)[\bm{\mathcal{T}}(0,0,w)-\bm{\mathcal{T}}(1,0,w)]+F_1(v)[\bm{\mathcal{T}}(0,1,w)-\bm{\mathcal{T}}(1,1,w)]\\
    &+F_0(w)[\bm{\mathcal{T}}(0,v,0)-\bm{\mathcal{T}}(1,v,0)]+F_1(w)[\bm{\mathcal{T}}(0,v,1)-\bm{\mathcal{T}}(1,v,1)]-F_0(v)F_0(w)[\bm{\mathcal{T}}(0,0,0)-\bm{\mathcal{T}}(1,0,0)]\\
    &-F_1(v)F_0(w)[\bm{\mathcal{T}}(0,1,0)-\bm{\mathcal{T}}(1,1,0)]-F_0(v)F_1(w)[\bm{\mathcal{T}}(0,0,1)-\bm{\mathcal{T}}(1,0,1)]\\
    &-F_1(v)F_1(w)[\bm{\mathcal{T}}(0,1,1)-\bm{\mathcal{T}}(1,1,1)],\\
     \bm  \lambda_2(u, w)&=\bm{\mathcal{S}}_4(u,w)-\bm{\mathcal{S}}_3(u,w)+F_0(u)[\bm{\mathcal{T}}(0,0,w)-\bm{\mathcal{T}}(0,1,w)]+F_1(u)[\bm{\mathcal{T}}(1,0,w)-\bm{\mathcal{T}}(1,1,w)]\\
     &+F_0(w)[\bm{\mathcal{T}}(u,0,0)-\bm{\mathcal{T}}(u,1,0)]+F_1(w)[\bm{\mathcal{T}}(u,0,1)-\bm{\mathcal{T}}(u,1,1)]-F_0(u)F_0(w)[\bm{\mathcal{T}}(0,0,0)-\bm{\mathcal{T}}(0,1,0)]\\
   &-F_1(u)F_0(w)[\bm{\mathcal{T}}(1,0,0)-\bm{\mathcal{T}}(1,1,0)]-F_0(u)F_1(w)[\bm{\mathcal{T}}(0,0,1)-\bm{\mathcal{T}}(0,1,1)]\\
   &-F_1(u)F_1(w)[\bm{\mathcal{T}}(1,0,1)-\bm{\mathcal{T}}(1,1,1)],\\
    \bm \lambda_3(u,v)&=\bm{\mathcal{S}}_2(u,v)-\bm{\mathcal{S}}_1(u,v)+F_0(u)[\bm{\mathcal{T}}(0,v,0)-\bm{\mathcal{T}}(0,v,1)]+F_1(u)[\bm{\mathcal{T}}(1,v,0)-\bm{\mathcal{T}}(1,v,1)]\\
    &+F_0(v)[\bm{\mathcal{T}}(u,0,0)-\bm{\mathcal{T}}(u,0,1)]+F_1(v)[\bm{\mathcal{T}}(u,1,0)-\bm{\mathcal{T}}(u,1,1)]-F_0(u)F_0(v)[\bm{\mathcal{T}}(0,0,0)-\bm{\mathcal{T}}(0,0,1)]\\
    &-F_1(u)F_0(v)[\bm{\mathcal{T}}(1,0,0)-\bm{\mathcal{T}}(1,0,1)]-F_0(u)F_1(v)[\bm{\mathcal{T}}(0,1,0)-\bm{\mathcal{T}}(0,1,1)]\\
    &-F_1(u)F_1(v)[\bm{\mathcal{T}}(1,1,0)-\bm{\mathcal{T}}(1,1,1)],
    \end{aligned}
     \label{eq:lambda1}
\end{equation}
and
\begin{equation}
    \begin{aligned}
        \bm \phi_1(u,v,w)&=F_0(v)\frac{\partial \bm{\mathcal{S}}_3(u,w)}{\partial u}+F_1(v)\frac{\partial \bm{\mathcal{S}}_4(u,w)}{\partial u}+F_0(w)\frac{\partial \bm{\mathcal{S}}_1(u,v)}{\partial u}+F_1(w)\frac{\partial\bm  {\mathcal{S}}_2(u,v)}{\partial u}\\
        &-F_0(v)F_0(w){\color{red} \partial_u \bm{\mathcal{T}}}(u,0,0) -F_1(v)F_0(w){\color{red} \partial_u \bm{\mathcal{T}}}(u,1,0)\\
        &-F_0(v)F_1(w){\color{red} \partial_u \bm{\mathcal{T}}}(u,0,1)-F_1(v)F_1(w){\color{red} \partial_u \bm{\mathcal{T}}}(u,1,1),\\
        \bm\phi_2(u,v,w)&=F_0(u)\frac{\partial \bm{\mathcal{S}}_5(v,w)}{\partial v}+F_1(u)\frac{\partial \bm{\mathcal{S}}_6(v,w)}{\partial v}+F_0(w)\frac{\partial \bm{\mathcal{S}}_1(u,v)}{\partial v}+F_1(w)\frac{\partial\bm  {\mathcal{S}}_2(u,v)}{\partial v}\\
       &-F_0(u)F_0(w){\color{red} \partial_v \bm{\mathcal{T}}}(0,v,0)-F_1(u)F_0(w){\color{red} \partial_v \bm{\mathcal{T}}}(1,v,0)\\
        &-F_0(u)F_1(w){\color{red} \partial_v \bm{\mathcal{T}}}(0,v,1)-F_1(u)F_1(w){\color{red} \partial_v \bm{\mathcal{T}}}(1,v,1),\\
      \bm  \phi_3(u,v,w)&=F_0(u)\frac{\partial \bm{\mathcal{S}}_5(v,w)}{\partial w}+F_1(u)\frac{\partial \bm{\mathcal{S}}_6(v,w)}{\partial w}+F_0(v)\frac{\partial \bm{\mathcal{S}}_3(u,w)}{\partial w}+F_1(v)\frac{\partial\bm  S_4(u,w)}{\partial w}\\
        &-F_0(u)F_0(v){\color{red} \partial_w \bm{\mathcal{T}}}(0,0,w)-F_1(u)F_0(v){\color{red} \partial_w \bm{\mathcal{T}}}(1,0,w)\\
       &-F_0(u)F_1(v){\color{red} \partial_w \bm{\mathcal{T}}}(0,1,w)-F_1(u)F_1(v){\color{red} \partial_w \bm{\mathcal{T}}}(1,1,w).
    \end{aligned}
     \label{eq:phi1}
\end{equation}


The Coons volume defines a smooth three-dimensional parameterization by interpolating its six boundary surfaces. However, such a mapping is not guaranteed to be regular in general. In the following, we demonstrate that the regularity of the mapping $\bm{\mathcal{T}}$ fundamentally depends on both the geometric variation of the boundary surfaces and their mutual compatibility. In Theorem~\ref{theorem 1}, we provide a rigorous theoretical analysis of how boundary geometry influences mapping regularity, and we establish a sufficient condition that ensures the Jacobian determinant of $\bm{\mathcal{T}}$ remains strictly positive throughout the domain. These results form the theoretical foundation for the construction of regular volumetric parameterizations in practical geometric modeling.

\subsection{Regularity analysis of Coons volume mapping in general form}
\label{sec202:Regularity Analysis of Coons Volume Mapping}

To analyze the regularity of the Coons volume mapping $\bm{\mathcal{T}}$, we first introduce several notations. Let us define
\begin{equation}
\begin{aligned}
    \rho = \max_{t \in [0,1]} |F_1'(t)|,
\end{aligned}
\end{equation}
\begin{equation}
\begin{aligned}
   {\color{red} \ h^1} &= \rho \max_{v, w \in [0,1]^2} \left \| \bm \lambda_1(v, w) \right \|,\\
   {\color{red} \ h^2}  &= \rho \max_{u, w \in [0,1]^2} \left \| \bm \lambda_2(u, w) \right \|,\\
   {\color{red} \ h^3} &= \rho \max_{u, v \in [0,1]^2} \left \| \bm \lambda_3(u, v) \right \|,
\end{aligned}
\end{equation}
and
\begin{equation}
    \begin{aligned}
     F = \max \left \{{\color{red} h^1,h^2,h^3} \right \}.
    \end{aligned}
    \label{eq:maxLambda}
\end{equation}
Let $M$ denote the maximum norm of the blending functions $\bm \phi_1, \bm \phi_2, \bm \phi_3$, i.e., 
\begin{equation}
\begin{aligned}
    M = \max \left \{ \max \| \phi_1(u,v,w)\|, \max \|\phi_2(u,v,w)\|, \max \|\phi_3(u,v,w)\| \right \}.
\end{aligned}
\label{eq:defM}
\end{equation}

\begin{theorem}\label{theorem 1}
Let $\bm{\mathcal{T}}$ be the Coons volume mapping defined over $[0,1]^3$. If there exists a constant $\tau > 0$ such that
\begin{equation}
    \det(\bm \phi_1, \bm \phi_2, \bm \phi_3) \geq \tau > 0,
    \label{eq:detPhi}
\end{equation}
and
\begin{equation}
    F^3 + 3MF^2 + 3FM^2 {\color{red}=0}< \tau,
    \label{eq:F3+3MF2+3FM2}
\end{equation}
then $\bm{\mathcal{T}}$ is regular throughout the domain.
\end{theorem}
\begin{proof}
From the definitions and norm estimations, we observe
\begin{equation}
\begin{aligned}
    F_1'(t) \| \bm \lambda_1(v,w) \| \leq F,\quad F_2'(t) \| \bm \lambda_2(u,w) \| \leq F,\quad F_3'(t) \| \bm \lambda_3(u,v) \| \leq F.
\end{aligned}
\label{eq:derivF}
\end{equation}

For notational simplicity, we denote
\begin{equation*}
    \begin{aligned}
       \bm \lambda_1 = \bm{\lambda}_1 (v,w), \quad
       \bm \lambda_2 = \bm \lambda_2(u,w), \quad
       \bm \lambda_3 = \bm \lambda_3(u,v).
    \end{aligned}
\end{equation*}
\begin{equation*}
    \begin{aligned}
        \tilde{\bm{\mathcal{T}}} = \det[\bm{\mathcal{T}}_u, \bm{\mathcal{T}}_v, \bm{\mathcal{T}}_w],\quad
        \Phi = \det [\bm \phi_1, \bm \phi_2, \bm \phi_3].
    \end{aligned}
\end{equation*}

Then, the Jacobian determinant of the Coons volume can be expanded to yield
\begin{equation} \label{det(T)}
\begin{aligned} 
    \tilde{\bm{\mathcal{T}}} &= F_1'(u) F_1'(v) F_1'(w) \det[\bm \lambda_1, \bm \lambda_2, \bm \lambda_3]+ F_1'(u) F_1'(v) \det [\bm \lambda_1, \bm \lambda_2,\bm  \phi_3]+ F_1'(u) F_1'(w) \det [\bm \lambda_1, \bm \phi_2, \bm \lambda_3] \\
    &+F_1'(u) \det [\bm \lambda_1, \bm \phi_2,\bm  \phi_3]+ F_1'(v) F_1'(w) \det [\bm \phi_1, \bm \lambda_2, \bm \lambda_3]+ F_1'(w) \det[\bm \phi_1, \bm \phi_2, \bm \lambda_3]+ F_1'(v) \det[\bm \phi_1,\bm  \lambda_2, \bm \phi_3]+ \Phi.
\end{aligned}
\end{equation}

Applying the triangle inequality and norm estimates (Eq.~\eqref{eq:derivF}), the lower bound estimation of Eq.~\eqref{det(T)} is expressed as
\begin{equation}
    \begin{aligned}
    \tilde{\bm{\mathcal{T}}} &\geq \Phi-(F_1'(u) F_1'(v) F_1'(w) \det[\bm \lambda_1,\bm  \lambda_2, \bm \lambda_3]+ F_1'(u) F_1'(v) \det [\bm \lambda_1, \bm \lambda_2, \bm \phi_3]+ F_1'(u) F_1'(w) \det [\bm \lambda_1,\bm  \phi_2, \bm \lambda_3]+ \\
      &F_1'(u) \det [\bm \lambda_1, \bm \phi_2, \bm \phi_3]+ F_1'(v) F_1'(w) \det [\bm \phi_1, \bm \lambda_2, \bm \lambda_3]+ F_1'(w) \det[\bm \phi_1, \bm \phi_2,\bm  \lambda_3]+ F_1'(v) \det[\bm \phi_1, \bm \lambda_2, \bm \phi_3]). 
    \end{aligned}
\end{equation}
By utilizing Eq.~\eqref{eq:defM} and Eq.~\eqref{eq:derivF}, a further estimation of $\tilde{\bm{\mathcal{T}}}$ demonstrate that 
\begin{equation} 
    \tilde{\bm{\mathcal{T}}} \geq \tau - \left( F^3 + 3MF^2 + 3FM^2 \right).
\end{equation}
Since $\Phi \geq \tau > 0$ and $F^3 + 3MF^2 + 3FM^2 < \tau$, we can {\color{red} obtain}  that 
\begin{equation}
    \det [ \bm{\mathcal{T}}_u, \bm{\mathcal{T}}_v, \bm{\mathcal{T}}_w ] = \tilde{\bm{\mathcal{T}}} > 0, \quad
    \forall (u,v,w) \in [0,1]^3.
\end{equation}
This ensures that the Jacobian determinant is strictly positive, implying the regularity of $\bm{\mathcal{T}}$.
\end{proof}

Building upon the general regularity condition established in Theorem~\ref{theorem 1}, we derive simplified sufficient conditions by considering reduced forms of the parametric volume representation or employing linear approximations to the blending functions \cite{randrianarivony2025deciding}. As an illustrative example, we examine the pair of boundary surfaces $\bm{\mathcal{S}}_5(v, w)$ and $\bm{\mathcal{S}}_6(v, w)$. Let $\bm{\mathcal{A}}(v, w)$ and $\bm{\mathcal{B}}(v, w)$ denote linear interpolation surfaces connecting the corresponding boundary curves of $\bm{\mathcal{S}}_5$ and $\bm{\mathcal{S}}_6$, defined as:
\begin{equation}
    \begin{aligned}
       \bm{\mathcal{A}}(v, w)= F_0(v) \bm{\mathcal{T}}(1,0,w) + F_1(v) \bm{\mathcal{T}}(1,1,w), \\
       \bm{\mathcal{B}}(v, w)= F_0(v) \bm{\mathcal{T}}(0,0,w) + F_1(v) \bm{\mathcal{T}}(0,1,w),
    \end{aligned}
\end{equation}
where $F_0(v)$ and $F_1(v)$ are standard linear blending functions.

Let $\Delta_\delta(v, w)$ and $\Delta_\beta(v, w)$ represent the deviations between the original and interpolated surfaces:
\begin{equation}
\begin{aligned}
    \Delta_\delta(v,w) &= \bm{\mathcal{S}}_5(v, w) - \bm{\mathcal{A}}(v, w), \\
    \Delta\beta(v,w) &= \bm{\mathcal{S}}_6(v, w) - \bm{\mathcal{B}}(v, w).
\end{aligned}
\label{eq:Delta_delta_beta}
\end{equation}

To analyze the derivative compatibility at corners, we introduce three auxiliary vector functions:
\begin{equation}
    \begin{aligned}
     \bm{\mathcal{H}}_1&=\frac{\partial \bm{\mathcal{S}}_3(u,w)}{\partial u}+\frac{\partial \bm{\mathcal{S}}_1(u,v)}{\partial u}-{\color{red} \partial_u \bm{\mathcal{T}}}(u,1,1),\\
     \bm{\mathcal{H}}_2&=\frac{\partial \bm{\mathcal{S}}_1(u,v)}{\partial v}+\frac{\partial \bm{\mathcal{S}}_5(v,w)}{\partial v}-{\color{red} \partial_v \bm{\mathcal{T}}}(1,v,1),\\
     \bm{\mathcal{H}}_3&=\frac{\partial \bm{\mathcal{S}}_5(v,w)}{\partial w}+\frac{\partial \bm{\mathcal{S}}_3(u,w)}{\partial w}-{\color{red} \partial_w \bm{\mathcal{T}}}(1,1,w) .
    \end{aligned}
\end{equation}

\begin{corollary} \label{corollary 1}
Suppose the boundary surfaces $\bm{\mathcal{S}}_i$ ($i=1,2,\ldots,6$) satisfy:
\begin{equation}
    \begin{aligned}
       \bm{\mathcal{S}}_5(v, w)= \bm{\mathcal{S}}_6(v, w)+ \bm{\mathbf{w}_1},\quad
       \bm{\mathcal{S}}_4(u, w) = \bm{\mathcal{S}}_3(u, w) + \bm{\mathbf{w}_2},\quad
       \bm{\mathcal{S}}_2(u,v)= \bm{\mathcal{S}}_1(u,v)+ \bm{\mathbf{w}_3},
    \end{aligned}
    \label{eq:face_shift}
\end{equation}
where $\mathbf{w}_1$, $\mathbf{w}_2$, $\mathbf{w}_3$ are constant vectors, as shown in Figure~\ref{fig2a:translation}, and each pair of opposite surfaces is non-intersecting, i.e.,
\begin{equation*}
    \begin{aligned}
     \text{Im}(\bm{\mathcal{S}}_5) \cap \text{Im}(\bm{\mathcal{S}}_6) = \emptyset, \quad 
     \text{Im}(\bm{\mathcal{S}}_3) \cap \text{Im}(\bm{\mathcal{S}}_4) = \emptyset, \quad
     \text{Im}(\bm{\mathcal{S}}_1)  \cap \text{Im}(\bm{\mathcal{S}}_2) = \emptyset.
    \end{aligned}
\end{equation*}
If there exists a constant $\tau > 0$ such that
\begin{equation}
    \begin{aligned}
     \det [\bm{\mathcal{H}}_1,\bm{\mathcal{H}}_2,\bm{\mathcal{H}}_3] > \tau,
     \end{aligned}
\end{equation}
then the Coons volume mapping $\bm{\mathcal{T}}$ constructed from these boundary surfaces is regular.
\end{corollary}

\begin{proof}
From Eq.~\eqref{eq:face_shift}, each opposing pair of surfaces is a rigid translation of the other. Consequently, the Coons blending terms defined in Eq.~\eqref{eq:phi1} satisfy:
\begin{equation}
    \begin{aligned}
       \bm \phi_1(u,v,w) = \bm{\mathcal{H}}_1,\quad
       \bm \phi_2(u,v,w) = \bm{\mathcal{H}}_2,\quad 
       \bm \phi_3(u,v,w) = \bm{\mathcal{H}}_3.
    \end{aligned}
\end{equation}
This implies
\begin{equation}
    \begin{aligned}
       \det [\bm \phi_1,\bm  \phi_2,\bm  \phi_3] = \det [\bm{\mathcal{H}}_1,\bm{\mathcal{H}}_2,\bm{\mathcal{H}}_3].
    \end{aligned}
\end{equation}
Meanwhile, from Eq.~\eqref{eq:lambda1} and Eq.~\eqref{eq:Delta_delta_beta}, the deviation vector field $\bm{\lambda}_1(v, w)$ becomes
\begin{equation}
    \begin{aligned}
      \bm{\lambda}_1(v,w)= \Delta_\delta(v,w) - \Delta_\beta(v,w).
     \end{aligned}
\end{equation}

Since $\bm{\mathcal{S}}_5$ and $\bm{\mathcal{S}}_6$ are planar or translations of each other, $\Delta_\delta = \Delta_\beta$, yielding $\bm{\lambda}_1(v, w) = \mathbf{0}$. Similar arguments apply to $\bm \lambda_2(u,w)=0$ and $\bm \lambda_3(u,v)=0$, thus the maximum norm $F = 0$. According to Eq.~\eqref{eq:maxLambda} in Theorem \ref{theorem 1}, it follows that
\begin{equation}
    F^3 + 3MF^2 + 3FM^2 {\color{red}=0} < \tau
\end{equation}
and
\begin{equation*}
  \det [ \bm{\phi}_1, \bm{\phi}_2, \bm{\phi}_3 ] > \tau > 0 . 
\end{equation*}
Therefore, by Theorem~\ref{theorem 1}, the mapping $\bm{\mathcal{T}}$ is regular.
\end{proof}

\begin{figure}[H]
    \centering
    \subfigure[]{
    \includegraphics[width=0.35\linewidth]{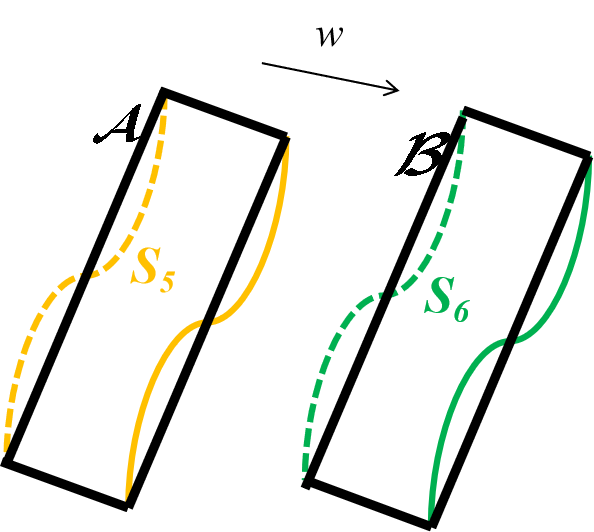}
    \label{fig2a:translation}
    }
    \quad
    \subfigure[]{
    \includegraphics[width=0.35\linewidth]{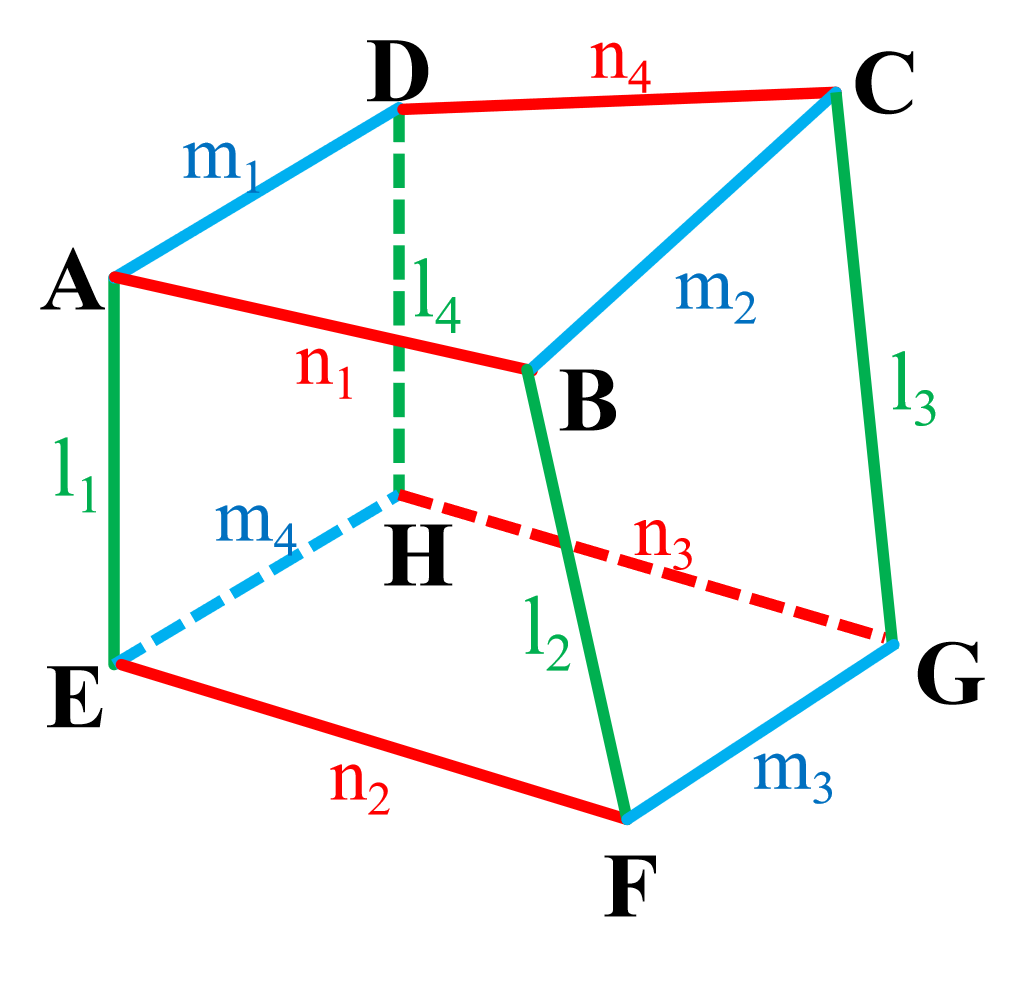}
    \label{fig2b:boundary structure}
    }
    \caption{Relative position of boundary surfaces and the boundary structure of the parametric volume: (a) Translation of a boundary surface. (b) Boundary structure of the parametric volume.}
    \label{fig2}
\end{figure}

\begin{corollary}\label{corollary 2}
Assume that all six boundary surfaces $\bm{\mathcal S}_i$ ($i=1,\ldots,6$) are bilinearly parameterized Coons patches, constructed from linear boundary curves and their corresponding corner vertices. Specifically, each surface admits a unified Coons blending representation of the form
{\color{red}
\begin{equation}
\begin{aligned}
    \bm{\mathcal S}_1(u,v) &= (1-u)n_3+u\,n_2 +(1-v)m_4+v\,m_3 
    -(1-u)(1-v)\mathbf{H} -u(1-v)\mathbf{E} -(1-u)v\mathbf{G} -uv\mathbf{F}, \\
    \bm{\mathcal S}_2(u,v) &= (1-u)n_4+u\,n_1 +(1-v)m_1+v\,m_2 
    -(1-u)(1-v)\mathbf{D}-u(1-v)\mathbf{A} -(1-u)v\mathbf{C} -uv\mathbf{B}, \\
    \bm{\mathcal S}_3(u,w) &= (1-u)l_4+u\,l_1 +(1-w)m_4+w\,m_1 
    -(1-u)(1-w)\mathbf{H}-u(1-w)\mathbf{E}-(1-u)w\mathbf{D}-uw\mathbf{A},\\
    \bm{\mathcal S}_4(u,w) &= (1-u)l_3+u\,l_2 +(1-w)m_3+w\,m_2 
    -(1-u)(1-w)\mathbf{G}-u(1-w)\mathbf{F}-(1-u)w\mathbf{C}-uw\mathbf{B}, \\
    \bm{\mathcal{S}}_5(v,w) &= (1-v)l_1 + v\,l_2 +(1-w)n_2+w\,n_1 
    -(1-v)(1-w)\mathbf{E}-v(1-w)\mathbf{F}-(1-v)w\mathbf{A}-vw\mathbf{B}, \\
    \bm{\mathcal S}_6(v,w) &= (1-v)l_4+v\,l_3 +(1-w)n_3+w\,n_4 
    -(1-v)(1-w)\mathbf{H}-v(1-w)\mathbf{G}-(1-v)w\mathbf{D}-vw\mathbf{C}.
\end{aligned}
\label{eq:S_surface_definition}
\end{equation}
}
Here, $l_i$, $m_i$, and $n_i$ denote linear boundary curves, and  $\mathbf A,\ldots,\mathbf H$ are the corner vertices of the hexahedral domain. The blending functions $F_0(t)$ and $F_1(t)$ are assumed to be linear. Under these assumptions, if one boundary surface, say $\bm{\mathcal S}_j$, forms a parallelogram such that the two boundary edges incident to any adjacent vertex are parallel and equal in length, then the resulting Coons volume mapping $\bm{\mathcal T}$ is regular, i.e., its Jacobian determinant does not vanish within the parameter domain.
\end{corollary}
\begin{proof}
Let the eight corner points of the parametric volume be denoted by $\mathbf{A}, \mathbf{B}, \mathbf{C}, \mathbf{D}, \mathbf{E}, \mathbf{F}, \mathbf{G}, \mathbf{H}$, and let the twelve edge vectors be represented by $l_i$, $n_j$, and $m_k$ for $i,j,k = 1,2,3,4$. Assume that the left boundary surface $\bm{\mathcal{S}}_3$, defined by the vertices $\mathbf{A}, \mathbf{D}, \mathbf{H}, \mathbf{E}$, forms a parallelogram. Furthermore, suppose the vectors $\overrightarrow{\mathbf{G}\mathbf{H}}$ and $\overrightarrow{\mathbf{F}\mathbf{E}}$ are parallel and of equal length, i.e., $\overrightarrow{\mathbf{G}\mathbf{H}} = \overrightarrow{\mathbf{F}\mathbf{E}}$, as illustrated in Figure~\ref{fig2b:boundary structure}.

Combining with Eq.~\eqref{eq:lambda1}, the expressions for the blending terms $\bm{\lambda}_i$ can be written as:
\begin{equation}
  \begin{aligned}
   \bm  \lambda_1(v, w)&=\bm{\mathcal{S}}_5(v,w)-\bm{\mathcal{S}}_6(v,w)+(1-v)(l_4-l_1)+v(l_3-l_2)+(1-w)(n_3-n_2)+w(n_4-n_1)\\
   &-(1-v)(1-w)(\mathbf{H}-\mathbf{E})-v(1-w)(\mathbf{G}-\mathbf{F})-(1-v)w(\mathbf{D}-\mathbf{A})-vw(\mathbf{C}-\mathbf{B}),  \\
   \bm  \lambda_2(u, w)&=\bm{\mathcal{S}}_4(u,w)-\bm{\mathcal{S}}_3(u,w)+(1-u)(l_4-l_3)+u(l_1-l_2)+(1-w)(m_4-m_3)+w(m_1-m_2)\\
   &-(1-u)(1-w)(\mathbf{H}-\mathbf{G})-u(1-w)(\mathbf{E}-\mathbf{F})-(1-u)w(\mathbf{D}-\mathbf{C})-uw(\mathbf{A}-\mathbf{B}),\\
    \bm \lambda_3(u,v)&=\bm{\mathcal{S}}_2(u,v)-\bm{\mathcal{S}}_1(u,v)+(1-u)(n_3-n_4)+u(n_2-n_1)+(1-v)(m_4-m_1)+v(m_3-m_2)\\
    &-(1-u)(1-v)(\mathbf{H}-\mathbf{D})-u(1-v)(\mathbf{E}-\mathbf{A})-(1-u)v(\mathbf{G}-\mathbf{C})-uv(\mathbf{F}-\mathbf{B}).
    \label{linear parameterization expressions}
    \end{aligned}
\end{equation}
Assuming that all six boundary surfaces $\bm{\mathcal{S}}_i$ ($i = 1,2,\ldots,6$) are bi-linearly parameterized and {\color{red} admit the same bilinear Coons--blending structure
as defined in Eq.~\eqref{eq:S_surface_definition}}, it follows directly that
\begin{equation}
\begin{aligned}
    \bm{\lambda}_1(v, w)= \bm{\lambda}_2(u, w)= \bm{\lambda}_3(u, v)=0.
\end{aligned}
\end{equation}
Hence, from the definition of the maximum norm $F$ in Eq.~\eqref{eq:maxLambda}, we have $F = 0$ and
\begin{equation}
    F^3 + 3MF^2 + 3FM^2 = 0 < \tau.
\end{equation}
which satisfies the sufficient condition~\eqref{eq:F3+3MF2+3FM2} in Theorem~\ref{theorem 1}.

To verify the positivity of the determinant involving the blending terms $\bm{\phi}_i$ ($i = 1,2,3$), we express the edge vectors of the boundary curves in terms of vertex differences:
\begin{equation}
\begin{aligned}
    l_1 &= \mathbf{A} -\mathbf{E} ,& l_2 &= \mathbf{B} - \mathbf{F} ,& l_3 &= \mathbf{C} - \mathbf{G} ,& l_4 &= \mathbf{D} - \mathbf{H}, \\
    n_1 &= \mathbf{B} -\mathbf{A} ,  & n_2 &= \mathbf{F} - \mathbf{E},  & n_3 &= \mathbf{G} - \mathbf{H},  & n_4 &= \mathbf{C} - \mathbf{D}, \\
    m_1 &= \mathbf{A}-\mathbf{D},   & m_2 &= \mathbf{B} - \mathbf{C},  & m_3 &= \mathbf{F} - \mathbf{G} , & m_4 &=\mathbf{E}  - \mathbf{H}.
\end{aligned}
\label{eq:defL}
\end{equation}
Substituting these into the expressions for $\bm{\mathcal{S}}_i$ and simplifying, we obtain
\begin{equation}
    \begin{aligned} 
     \bm{\mathcal{S}}_1(u,v)&=(1-v-2u+uw)\mathbf{E}+(u+v-uv)\mathbf{F}+(2u+2v-uv-3)\mathbf{H}+(1-u-2v+uv)\mathbf{G}, \\
     \bm{\mathcal{S}}_2(u,v)&=(1-v-2u+uw)\mathbf{A}+(u+v-uv)\mathbf{B}+(2u+2v-uv-3)\mathbf{D}+(1-u-2v+uv)\mathbf{C},\\
      \bm{\mathcal{S}}_3(u,w)&=(1-w-2u+uw)\mathbf{E}+(u+w-uw)\mathbf{A}+(2u+2w-uw-3)\mathbf{H}+(1-u-2w+uw)\mathbf{D}, \\
     \bm{\mathcal{S}}_4(u,w)&=(1-w-2u+uw)\mathbf{F}+(u+w-uw)\mathbf{B}+(2u+2w-uw-3)\mathbf{G}+(1-u-2w+uw)\mathbf{C}, \\
     \bm{\mathcal{S}}_5(v,w)&=(1-v-2w+vw)\mathbf{A}+(v+w-vw)\mathbf{B}+(2v+2w-vw-3)\mathbf{E}+(1-w-2v+vw)\mathbf{F},\\
     \bm{\mathcal{S}}_6(v,w)&=(1-v-2w+vw)\mathbf{D}+(v+w-vw)\mathbf{C}+(2v+2w-vw-3)\mathbf{H}+(1-w-2v+vw)\mathbf{G}.
      \end{aligned}
      \label{eq:S5linear}
\end{equation}
Then, taking appropriate boundary values of $0$ or $1$ for the parameter $u,v,w$ in  $\bm{\mathcal{S}}_i$ ($i=1,...,6$), respectively, the relationship between the tangent directions of the boundary curves and the vectors formed by the vertices of the parameter volume can be obtained, i.e.,

\begin{equation}
    \begin{alignedat}{4}
       {\color{red} \partial_u \bm{\mathcal{T}}}(u,0,0)   &= \overrightarrow{\mathbf{GF}} + 2\overrightarrow{\mathbf{EH}}, 
            \quad & {\color{red} \partial_u \bm{\mathcal{T}}}(u,1,0) &=   {\color{red} \partial_u \bm{\mathcal{T}}}(u,0,1) = \overrightarrow{\mathbf{EH}} = \overrightarrow{\mathbf{CB}} + 2\overrightarrow{\mathbf{FG}}, 
            \quad &  {\color{red} \partial_u \bm{\mathcal{T}}}(u,1,1) &= \overrightarrow{\mathbf{GF}} = \overrightarrow{\mathbf{DA}}, \\
        {\color{red} \partial_v \bm{\mathcal{T}}}(0,v,0)   &= \overrightarrow{\mathbf{DC}} + 2\overrightarrow{\mathbf{GH}}, 
            \quad &  {\color{red} \partial_v \bm{\mathcal{T}}} (1,v,0) &=  {\color{red} \partial_v \bm{\mathcal{T}}}(0,v,1) = \overrightarrow{\mathbf{GH}} = \overrightarrow{\mathbf{AB}} + 2\overrightarrow{\mathbf{FE}}, 
            \quad &  {\color{red} \partial_v \bm{\mathcal{T}}}(1,v,1) &= \overrightarrow{\mathbf{FE}} = \overrightarrow{\mathbf{CD}}, \\
       {\color{red} \partial_w \bm{\mathcal{T}}}(0,0,w)   &= \overrightarrow{\mathbf{GC}} + 2\overrightarrow{\mathbf{DH}} ,
            \quad &  {\color{red} \partial_w \bm{\mathcal{T}}}(0,1,w) &=   {\color{red} \partial_w \bm{\mathcal{T}}}(1,0,w) = \overrightarrow{\mathbf{DH}} = \overrightarrow{\mathbf{FB}} + 2\overrightarrow{\mathbf{CG}}, 
            \quad &  {\color{red} \partial_w \bm{\mathcal{T}}}(1,1,w) &= \overrightarrow{\mathbf{CG}} = \overrightarrow{\mathbf{AE}}.
    \end{alignedat}
    \label{eq:S1-6}
\end{equation}

By computing partial derivatives for all boundary surfaces $\bm{\mathcal{S}}_i$ $(i=1, ...,6)$, we obtain
\begin{equation}
    \begin{alignedat}{2}
        \frac{\partial \bm{\mathcal{S}}_1(u,v)}{\partial u} &= (1-v)\overrightarrow{\mathbf{GF}}+(2-v) \overrightarrow{\mathbf{EH}},
        &\quad
        \frac{\partial \bm{\mathcal{S}}_1(u,v)}{\partial v} &= (1-u)\overrightarrow{\mathbf{EF}}+(2-u) \overrightarrow{\mathbf{GH}},\\
        \frac{\partial \bm{\mathcal{S}}_2(u,v)}{\partial u} &= (1-v)\overrightarrow{\mathbf{CB}}+(2-v) \overrightarrow{\mathbf{AD}},
        &\quad
        \frac{\partial \bm{\mathcal{S}}_2(u,v)}{\partial v} &= (1-u)\overrightarrow{\mathbf{AB}}+(2-u) \overrightarrow{\mathbf{CD}},\\
        \frac{\partial \bm{\mathcal{S}}_3(u,w)}{\partial u} &= (1-w)\overrightarrow{\mathbf{DA}}+(2-w) \overrightarrow{\mathbf{EH}},
        &\quad
        \frac{\partial \bm{\mathcal{S}}_3(u,w)}{\partial w} &= (1-u)\overrightarrow{\mathbf{EA}}+(2-u) \overrightarrow{\mathbf{DH}},\\
        \frac{\partial \bm{\mathcal{S}}_4(u,w)}{\partial u} &= (1-w)\overrightarrow{\mathbf{CB}}+(2-w) \overrightarrow{\mathbf{FG}},
        &\quad
        \frac{\partial \bm{\mathcal{S}}_4(u,w)}{\partial w} &= (1-u)\overrightarrow{\mathbf{FB}}+(2-u) \overrightarrow{\mathbf{CG}},\\
        \frac{\partial \bm{\mathcal{S}}_5(v,w)}{\partial v} &= (1-w)\overrightarrow{\mathbf{AB}}+(2-w) \overrightarrow{\mathbf{FE}},
        &\quad
        \frac{\partial \bm{\mathcal{S}}_5(v,w)}{\partial w} &= (1-v)\overrightarrow{\mathbf{FB}}+(2-v) \overrightarrow{\mathbf{AE}},\\
        \frac{\partial \bm{\mathcal{S}}_6(v,w)}{\partial v} &= (1-w)\overrightarrow{\mathbf{DC}}+(2-w) \overrightarrow{\mathbf{GH}},
        &\quad
        \frac{\partial \bm{\mathcal{S}}_6(v,w)}{\partial w} &= (1-v)\overrightarrow{\mathbf{GC}}+(2-v) \overrightarrow{\mathbf{DH}}.
    \end{alignedat}
    \label{eq:parS}
\end{equation}

Substituting Eq.~\eqref{eq:S1-6} and Eq.~\eqref{eq:parS} into Eq.~\eqref{eq:phi1}, the blending terms $\bm \phi_i(u,v,w)(i=1,2,3)$ can be simplified as
\begin{equation}
    \begin{aligned}
       \bm  \phi_1(u,v,w)&=(2-v-w)\overrightarrow{\mathbf{EH}}+(v+w-1)\overrightarrow{\mathbf{AD}},\\
        \bm \phi_2(u,v,w)&=(1-w+u)\overrightarrow{\mathbf{GH}}+(w-u)\overrightarrow{\mathbf{FE}},\\
      \bm  \phi_3(u,v,w)&=(1-v+u)\overrightarrow{\mathbf{DH}}+(v-u)\overrightarrow{\mathbf{AE}}.
    \end{aligned}
    \label{eq:phi2_3_linear}
\end{equation}

Using the fact that the Coons volume $\bm{\mathcal{T}}$ is convex and that the left face $\bm{\mathcal{S}}_3$ is a parallelogram with $\overrightarrow{\mathbf{GH}} = \overrightarrow{\mathbf{FE}}$, we compute the determinant
\begin{equation}
\begin{aligned}
  \det[\bm\phi_1, \bm\phi_2, \bm\phi_3]&=(2-v-w)(1-w+u)(1-v+u) \det[\overrightarrow{\mathbf{EH}},\overrightarrow{\mathbf{GH}},\overrightarrow{\mathbf{DH}}]\\
  &+(2-v-w)(1-w+u)(v-u) \det[\overrightarrow{\mathbf{EH}},\overrightarrow{\mathbf{GH}},\overrightarrow{\mathbf{AE}}]\\
   &+(2-v-w)(w-u)(1-v+u)\det[\overrightarrow{\mathbf{EH}},\overrightarrow{\mathbf{FE}},\overrightarrow{\mathbf{DH}}]\\
   &+(2-v-w)(w-u)(v-u)\det[\overrightarrow{\mathbf{EH}},\overrightarrow{\mathbf{FE}},\overrightarrow{\mathbf{AE}}]\\
   &+(v+w-1)(1-w+u)(1-v+u) \det[\overrightarrow{\mathbf{AD}},\overrightarrow{\mathbf{GH}},\overrightarrow{\mathbf{DH}}]\\
  &+(v+w-1)(1-w+u)(v-u) \det[\overrightarrow{\mathbf{AD}},\overrightarrow{\mathbf{GH}},\overrightarrow{\mathbf{AE}}]\\
   &+(v+w-1)(w-u)(1-v+u)\det[\overrightarrow{\mathbf{AD}},\overrightarrow{\mathbf{FE}},\overrightarrow{\mathbf{DH}}]\\
   &+(v+w-1)(w-u)(v-u)\det[\overrightarrow{\mathbf{AD}},\overrightarrow{\mathbf{FE}},\overrightarrow{\mathbf{AE}}]\\
   &= \det[\overrightarrow{\mathbf{EH}},\overrightarrow{\mathbf{GH}},\overrightarrow{\mathbf{DH}}]>0.
\end{aligned}
\end{equation}
This confirms that the regularity condition in Eq.~\eqref{eq:detPhi} of Theorem~\ref{theorem 1} is satisfied. Therefore, the Coons volume mapping $\bm{\mathcal{T}}$ is regular.
\end{proof}

In Corollaries~\ref{corollary 1} and~\ref{corollary 2}, we establish sufficient conditions under which the Coons volume mapping $\bm{\mathcal{T}}$ maintains regularity, based on specific geometric configurations and linear boundary conditions. Corollary~\ref{corollary 1} shows that if each pair of opposite boundary surfaces is related by a strict translation and their associated derivative structures are consistent, then the resulting Coons volume $\bm{\mathcal{T}}$ remains regular across the entire parametric domain. Corollary~\ref{corollary 2} further asserts that when both the boundary surfaces and the blending functions are linearly parameterized, the Coons mapping naturally satisfies the condition of a strictly positive Jacobian determinant throughout the domain, thereby guaranteeing its regularity and invertibility. These theoretical results offer both geometric intuition and practical criteria for analyzing the structural regularity of Coons parametric volumes. {\color{red} A similar linear condition has been derived in~\cite{kosinka2010injective} in the context of cage-based free-form deformation.}

To validate these theoretical findings, we designed two representative test cases corresponding to the configurations described in Corollaries~\ref{corollary 1} and~\ref{corollary 2}, respectively. Figure~\ref{fig:coons-jacobian-regularity} illustrates the deformation of the volume grid and the spatial distribution of the Jacobian determinant for the constructed Coons mappings under the respective geometric settings. It is observed that when the geometric conditions specified in either Corollary~\ref{corollary 1} or Corollary~\ref{corollary 2} are satisfied, the Coons volume mapping $\bm{\mathcal{T}}$ remains smooth and free from self-intersections, with the Jacobian determinant strictly positive across the domain. These numerical results provide evidence supporting the theoretical regularity conditions established in both corollaries.

\begin{figure}[H]
    \centering
    \subfigure[]{\includegraphics[width=0.25\linewidth]{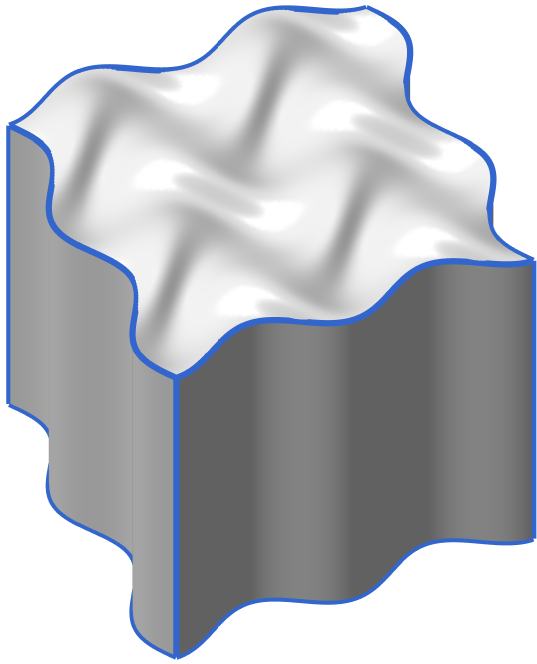}
    \label{fig4a_similarityExample}}
    \quad
    \subfigure[]{\includegraphics[width=0.35\linewidth]{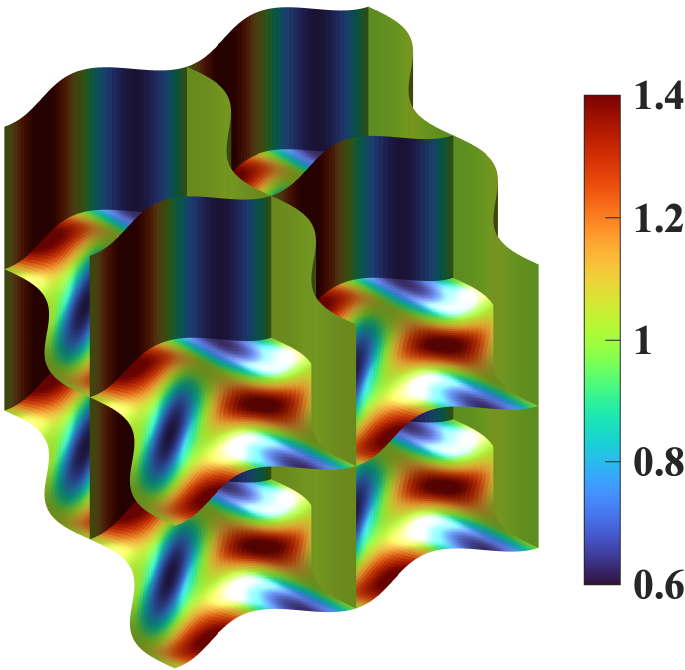}
    \label{fig4b_similarityJacobian}}
    \\
    \subfigure[]{\includegraphics[width=0.25\linewidth]{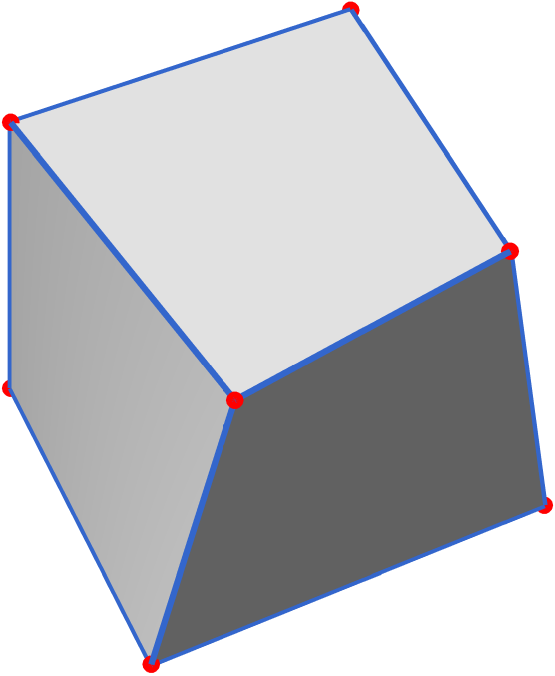}
    \label{fig4c_linearExample}}
    \quad
    \subfigure[]{\includegraphics[width=0.35\linewidth]{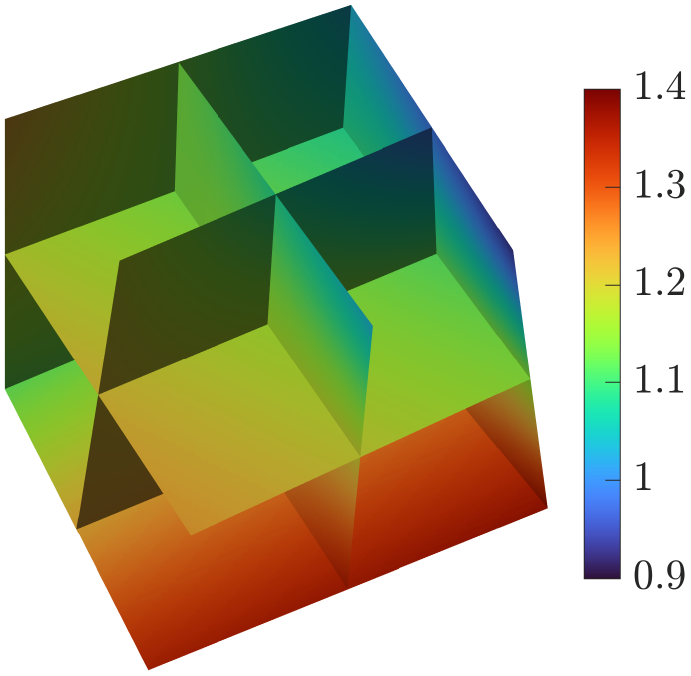}
    \label{fig4d_linearExampleJacobian}}
    \caption{Regularity verification of Coons volume mappings under two representative boundary configurations. 
    (a), (b): Mapping result and corresponding Jacobian determinant for Corollary~\ref{corollary 1}, where the six boundary surfaces exhibit translational relationships. 
    (c), (d): Mapping result and Jacobian determinant for Corollary~\ref{corollary 2}, where both boundary surfaces and blending functions are linearly parameterized. 
    In both cases, the Jacobian determinants remain strictly positive across the domain, confirming the bijectivity and regularity of the constructed Coons volume mappings.}
    \label{fig:coons-jacobian-regularity}
\end{figure}

Theorem~\ref{theorem 1} provides a sufficient condition for ensuring the regularity of Coons volume mappings constructed from general boundary surfaces. Notably, this condition is independent of the specific parametric representations of the boundary surfaces or the form of the blending functions, thus exhibiting strong generality and applicability across a wide range of geometric configurations and parameter domains. The regularity of the mapping is primarily determined by the geometric characteristics of the boundary surfaces, particularly the magnitude of surface variation and the degree of similarity between opposing boundary faces. By imposing appropriate constraints on these geometric features, we further derive simplified sufficient conditions for specific structural cases, as presented in Corollaries~\ref{corollary 1} and~\ref{corollary 2}. These cases retain the simplicity of the Coons construction while guaranteeing the regularity of the resulting Coons volume mapping $\bm{\mathcal{T}}$.

\subsection{Regularity analysis of Coons volume mappings in B\'ezier formulation}
\label{sec203:Regularity Analysis of Coons Volume Mappings in B\'ezier Representation}

The B\'ezier method is one of the most widely adopted parametric representations in engineering design and computational geometry due to its intuitive geometric control and mathematical simplicity. In many practical applications, both the boundary surfaces and the blending functions of Coons volume constructions are expressed in B\'ezier form.

For this commonly encountered case, we reformulate the general regularity condition proposed in Theorem~\ref{theorem 1} into a set of constraints directly imposed on the B\'ezier control points. This reformulation enables a more intuitive geometric interpretation of the regularity condition and facilitates its practical verification in computational implementations.

Suppose that all six boundary surfaces of the Coons volume are represented in B\'ezier form \cite{farin2014curves}. In this setting, the boundary surfaces can be expressed as follows
\begin{equation} 
\begin{aligned}
     \bm{\mathcal{T}}(0, v, w) &=  \sum_{j=0}^{n} \sum_{k=0}^{n} \bm{P}_{0jk} \mathcal{B}_j^n(v) \mathcal{B}_k^n(w) , \quad
     \bm{\mathcal{T}}(1, v, w) =  \sum_{j=0}^{n} \sum_{k=0}^{n} \bm{P}_{njk} \mathcal{B}_j^n(v) \mathcal{B}_k^n(w),\\
     \bm{\mathcal{T}}(u, 0, w) &=  \sum_{i=0}^{n} \sum_{k=0}^{n} \bm{P}_{i0k} \mathcal{B}_i^n(u) \mathcal{B}_k^n(w),\quad\bm{\mathcal{T}}(u, 1, w) =  \sum_{i=0}^{n} \sum_{k=0}^{n} \bm{P}_{ink} \mathcal{B}_i^n(u) \mathcal{B}_k^n(w) , \\
     \bm{\mathcal{T}}(u, v, 0)& =  \sum_{i=0}^{n} \sum_{j=0}^{n} \bm{P}_{ij0} \mathcal{B}_i^n(u) \mathcal{B}_j^n(v),\quad\bm{\mathcal{T}}(u, v, 1)=  \sum_{i=0}^{n} \sum_{j=0}^{n} \bm{P}_{ijn} \mathcal{B}_i^n(u) \mathcal{B}_j^n(v),
\end{aligned}
\end{equation}
where $\mathcal{B}_i^n(u)$, $\mathcal{B}_j^n(v)$, and $\mathcal{B}_k^n(w)$ represent the B\'ezier basis functions of degree $n$, $\bm{P}_{ijk}$ denote the control points, and the blending functions $F_0(t)$ and $F_1(t)$ are also expressed in B\'ezier form, namely
\begin{equation}
\begin{aligned}
   F_1(t)& = 1-F_0(t) = \sum_{i=0}^{n} \alpha_i \mathcal{B}_i^n(t),
\end{aligned}
\end{equation}
{\color{red}where $\alpha_i$ denote the scalar coefficients of the B\'ezier representation of the blending functions.} Under these assumptions, the Coons volume mapping $\bm{\mathcal{T}}(u, v, w)$ can be represented as a trivariate B\'ezier volume
\begin{equation}
\begin{aligned}
    \bm{\mathcal{T}}(u, v, w) &= \sum_{i=0}^{n} \sum_{j=0}^{n} \sum_{k=0}^{n} \bm{P}_{ijk} \mathcal{B}_i^n(u) \mathcal{B}_j^n(v) \mathcal{B}_k^n(w),
\end{aligned}
\end{equation} 
where the control points $\bm{P}_{ijk}$ is given by
\begin{equation}
\begin{aligned}
    \bm{P}_{ijk}&=\bm{P}_{0jk}+\bm{P}_{i0k}+\bm{P}_{ij0}-\bm{P}_{i00}-\bm{P}_{0j0}-\bm{P}_{00k}+\bm{P}_{000}+\alpha_i[\bm{P}_{njk}-\bm{P}_{0jk}+\bm{P}_{00k}-\bm{P}_{n0k}+\bm{P}_{0j0}-\bm{P}_{nj0}\\
    &-\bm{P}_{000}+\bm{P}_{n00}+\alpha_j(\bm{P}_{n0k}-\bm{P}_{00k}+\bm{P}_{0nk}-\bm{P}_{nnk}+\bm{P}_{000}-\bm{P}_{0n0}-\bm{P}_{n00}+\bm{P}_{nn0})+\alpha_k(\bm{P}_{0jn}-\bm{P}_{0j0}\\
    &+\bm{P}_{nj0}-\bm{P}_{njn}+\bm{P}_{000}-\bm{P}_{n00}-\bm{P}_{00n}+\bm{P}_{n0n})+\alpha_j\alpha_k(\bm{P}_{0n0}-\bm{P}_{000}+\bm{P}_{n00}+\bm{P}_{00n}-\bm{P}_{nn0}-\bm{P}_{0nn}\\
    &-\bm{P}_{n0n}+\bm{P}_{nnn})]+\alpha_j[\bm{P}_{ink}-\bm{P}_{i0k}+\bm{P}_{00k}-\bm{P}_{0nk}-\bm{P}_{in0}+\bm{P}_{0n0}-\bm{P}_{000}+\alpha_k(\bm{P}_{in0}-\bm{P}_{i00}+\bm{P}_{i0n}\\
    &-\bm{P}_{inn}+\bm{P}_{000}-\bm{P}_{0n0}-\bm{P}_{00n}+\bm{P}_{0nn})]+\alpha_k(\bm{P}_{ijn}-\bm{P}_{ij0}+\bm{P}_{i00}+\bm{P}_{0j0}-\bm{P}_{0jn}-\bm{P}_{000}+\bm{P}_{00n}).
\end{aligned}
\end{equation}

\begin{theorem}\label{theorem 2}
Let $\bm{\mathcal{T}} : [0,1]^3 \rightarrow \mathbb{R}^3$ be a trivariate Coons mapping of degree $(n,n,n)$ expressed in B\'ezier form, with control points $\bm{P}_{ijk}$. If for all multi-indices $p,q,r = 0,1,\dots,3n-1$, the following inequality holds
\begin{equation}
    \bm{\mathcal{J}}_{pqr} = \sum_{\substack{i_1+i_2+i_3=p \\ j_1+j_2+j_3=q \\ k_1+k_2+k_3=r}} D(i_1, j_1, k_1 , i_2, j_2, k_2 , i_3 ,j_3 ,k_3) \frac{\binom{n-1}{i_1} \binom{n}{i_2} \binom{n}{i_3}  \binom{n}{j_1}\binom{n-1}{j_2}\binom{n}{j_3}  \binom{n}{k_1} \binom{n}{k_2} \binom{n-1}{k_3}}{\binom{3n-1}{i_1+j_1+k_1} \binom{3n-1}{i_2+j_2+k_2} \binom{3n-1}{i_3+j_3+k_3}} > 0,
    \label{eq:Jpqr}
\end{equation}
where
\begin{equation}
    D(i_1, j_1, k_1, i_2, j_2, k_2, i_3, j_3, k_3) = n^3 \det \left[ \bm{P}_{i_1+1, j_1, k_1} - \bm{P}_{i_1, j_1, k_1}, \ \bm{P}_{i_2, j_2+1, k_2} - \bm{P}_{i_2, j_2, k_2}, \ \bm{P}_{i_3, j_3, k_3+1} - \bm{P}_{i_3, j_3, k_3} \right],
    \label{eq:computeD}
\end{equation}
then the Coons volume mapping $\bm{\mathcal{T}}$ is regular on $[0,1]^3$.
\end{theorem}

\begin{proof}
Assume that the three-dimensioal Coons mapping $\bm{\mathcal{T}}$ is given in B\'ezier form
\begin{equation}
\begin{aligned}
    \bm{\mathcal{T}}(u, v, w) = \sum_{i=0}^{n} \sum_{j=0}^{n} \sum_{k=0}^{n} \mathbf{P}_{i,j,k} \mathcal{B}_i^n(u) \mathcal{B}_j^{n}(v) \mathcal{B}_k^{n}(w).
    \end{aligned}
\end{equation}
Its first-order partial derivatives are
\begin{equation}
\begin{aligned}
   {\color{red} \partial_u \bm{\mathcal{T}}}(u, v, w) &= \sum_{i_1=0}^{n-1} \sum_{j_1=0}^{n} \sum_{k_1=0}^{n} n (\mathbf{P}_{i+1,j,k} - \mathbf{P}_{i,j,k}) \mathcal{B}_{i_1}^{n-1}(u) \mathcal{B}_{j_1}^{n}(v) \mathcal{B}_{k_1}^{n}(w) ,\\
  {\color{red} \partial_v \bm{\mathcal{T}}}(u, v, w) &= \sum_{i_2=0}^{n} \sum_{j_2=0}^{n-1} \sum_{k_2=0}^{n} n (\mathbf{P}_{i,j+1,k} - \mathbf{P}_{i,j,k}) \mathcal{B}_{i_2}^{n}(u) \mathcal{B}_{j_2}^{n-1}(v) \mathcal{B}_{k_2}^{n}(w), \\
   {\color{red} \partial_w \bm{\mathcal{T}}}(u, v, w) &= \sum_{i_3=0}^{n} \sum_{j_3=0}^{n} \sum_{k_3=0}^{n-1} n (\mathbf{P}_{i,j,k+1} - \mathbf{P}_{i,j,k}) \mathcal{B}_{i_3}^{n}(u) \mathcal{B}_{j_3}^{n}(v) \mathcal{B}_{k_3}^{n-1}(w).
\end{aligned}
\end{equation}
The Jacobian determinant $\tilde{\bm{\mathcal{T}}}$ is thus expressed as
\begin{equation}
    \begin{aligned}
       \tilde{\bm{\mathcal{T}}} &= \sum_{i_1=0}^{n-1}  \sum_{i_2,i_3=0}^{n} \sum_{j_2=0}^{n-1}\sum_{j_1,j_3=0}^{n} \sum_{k_3=0}^{n-1} \sum_{k_1,k_2=0}^{n} D \mathcal{B}_{i_1,j_1,k_1}^{n-1,n,n}(u,v,w) \mathcal{B}_{i_2,j_2,k_2}^{n,n-1,n}(u,v,w) \mathcal{B}_{i_3,j_3,k_3}^{n,n,n-1}(u,v,w),\\
    \end{aligned}
\end{equation}
where ${\color{red}D=D(i_1, j_1, k_1, i_2, j_2, k_2, i_3, j_3, k_3) }$ is defined by Eq.~\eqref{eq:computeD}. By utilizing the product identity of Bernstein basis functions
\begin{equation}
    \mathcal{B}_i^{n}(u)\mathcal{B}_j^{m}(u) = \frac{\binom{n}{i} \binom{m}{j}}{\binom{n+m}{i+j}} \mathcal{B}_{i+j}^{n+m}(u),
\end{equation}
we can express the Jacobian determinant $\tilde{\bm{\mathcal{T}}}$ as a linear combination of Bernstein tensor-product basis functions
\begin{equation}
    \left | \bm{\mathcal{J}}(u, v, w) \right | = \tilde{\bm{\mathcal{T}}}= \sum_{p=0}^{3n-1} \sum_{q=0}^{3n-1} \sum_{r=0}^{3n-1}\bm{\mathcal{J}}_{pqr} \mathcal{B}_p^{3n-1}(u) \mathcal{B}_q^{3n-1}(v) \mathcal{B}_r^{3n-1}(w) .
\end{equation}

Since Bernstein basis functions $\mathcal{B}_i^{n}(u)$ are non-negative and form a partition of unity, the positivity of all coefficients $\bm{\mathcal{J}}_{pqr}$ implies
\begin{equation}
    \bm{\mathcal{J}}(u, v, w) > 0, \quad \forall (u, v, w) \in [0,1]^3.
\end{equation}
Consequently, the Coons volume mapping $\bm{\mathcal{T}}$ is regular.
\end{proof}

In Theorems~\ref{theorem 1} and~\ref{theorem 2}, we established sufficient conditions for the regularity of Coons volume $\bm{\mathcal{T}}$ when both the boundary surfaces and the blending functions are expressed in general forms or B\'ezier forms, respectively. These sufficient conditions reveal the relationship between the regularity of the Coons mapping and its geometric constructions. Building on these foundations, we now introduce a subdivision-based analysis and propose a necessary condition for the regularity of $\bm{\mathcal{T}}$, formulated in terms of the sign consistency of the Jacobian determinant coefficients over the B\'ezier sub-blocks. 

{\color{red}Related derivations based on B\'ezier subdivision and Jacobian sign conditions have been previously reported in~\cite{gain2002preventing} in the context of free-form deformation.}

\begin{theorem}\label{theorem 3}
Let $\bm{\mathcal{T}}$ be a regular Coons volume mapping defined over $[0,1]^3$, bounded by six boundary surfaces $\bm{\mathcal{S}}_i$ $(i = 1,2,\dots,6)$, and assume its Jacobian determinant $\bm{\mathcal{J}}(u,v,w)$ satisfies $\bm{\mathcal{J}} > 0$ everywhere. Partition $\bm{\mathcal{T}}$ uniformly into $\sigma$ subdivisions along each parametric direction, resulting in $\sigma^3$ B\'ezier sub-blocks $\bm{\mathcal{T}}^{ijk}$, for $i,j,k = 1,2,\dots,\sigma$. Then, for sufficiently large $\sigma$, the B\'ezier coefficients $\bm{\mathcal{J}}_{pqr}^{ijk}$ $(p,q,r = 0,1,\ldots,3n-1)$ of the Jacobian determinant in each sub-block $\bm{\mathcal{T}}^{ijk}$ must all share the same sign.
\end{theorem}

\begin{proof}
Since $\bm{\mathcal{T}}$ is regular, its Jacobian determinant $\bm{\mathcal{J}}(u,v,w)$ is continuous on the closed domain $[0,1]^3$. Therefore, there exists a constant $\rho > 0$ such that
\begin{equation}
    \bm{\mathcal{J}}(u,v,w) \geq \rho, \quad \forall (u,v,w) \in [0,1]^3.
\end{equation}

Fix a sub-block index $(i,j,k)$ and consider the restriction of the Jacobian determinant on this subdomain, denoted by $\bm{\mathcal{J}}^{ijk}(u,v,w)$. According to B\'ezier theory, this function can be represented using the trivariate blossom function $\mathbf{P}^{ijk}$
\begin{equation}
    \bm{\mathcal{J}}^{ijk} = \mathbf{P}^{ijk}(\underbrace{u,\dots,u}_{3n-1}, \underbrace{v,\dots,v}_{3n-1}, \underbrace{w,\dots,w}_{3n-1}).
\end{equation}

Let $h = \frac{1}{(3n-1)\sigma}$ denote the grid spacing, and define the evaluation points
\begin{equation}
    a_p = a + ph, \quad c_q = c + qh, \quad e_r = e + rh, \quad p, q, r = 0, 1, \ldots, 3n-1.
\end{equation}

Next, by performing the first-order multivariate Taylor expansion of the blossom function $\mathbf{P}^{ijk}$ at the point $(\underbrace{a_p, \dots, a_p}_{3n-1}, \underbrace{c_q, \dots, c_q}_{3n-1}, \underbrace{e_r, \dots, e_r}_{3n-1})$ and utilizing the symmetry property of the blossom functions, we obtain
\begin{equation}
    \begin{aligned}
   &\mathbf{P}^{ijk}(\underbrace{a,...,a}_{3n-1-p}, \underbrace{b,...,b}_{p}; \underbrace{c,...,c}_{3n-1-q}, \underbrace{d,...,d}_{q}; \underbrace{e,...,e}_{3n-1-r}, \underbrace{f,...,f}_{r}) = \mathbf{P}^{ijk}(\underbrace{a_p,...,a_p}_{3n-1},\underbrace{c_q,...,c_q}_{3n-1},\underbrace{e_r,...,e_r}_{3n-1}) \\
   &+ \sum_{s=1}^{3n-1-p} (a - a_p) \frac{\partial}{\partial u_s} \mathbf{P}^{ijk}(\underbrace{a_p,...,a_p}_{3n-1}, \underbrace{c_q,...,c_q}_{3n-1},\underbrace{e_r,...,e_r}_{3n-1}) + \sum_{s=3n-p}^{3n-1} (b - a_p) \frac{\partial}{\partial u_s}  \mathbf{P}^{ijk}(\underbrace{a_p,...,a_p}_{3n-1}, \underbrace{c_q,...,c_q}_{3n-1}, \underbrace{e_r,...,e_r}_{3n-1}) \\
   &+ \sum_{s=1}^{3n-q-1} (c - c_q) \frac{\partial}{\partial v_s} \mathbf{P}^{ijk}(\underbrace{a_p,...,a_p}_{3n-1}, \underbrace{c_q,...,c_q}_{3n-1}, \underbrace{e_r,...,e_r}_{3n-1}) + \sum_{s=3n-q}^{3n-1} (d - c_q) \frac{\partial}{\partial v_s} \mathbf{P}^{ijk}(\underbrace{a_p,...,a_p}_{3n-1}, \underbrace{c_q,...,c_q}_{3n-1}, \underbrace{e_r,...,e_r}_{3n-1}) \\
   &+ \sum_{s=1}^{3n-1-r} (e - e_r) \frac{\partial}{\partial w_s} \mathbf{P}^{ijk}(\underbrace{a_p,...,a_p}_{3n-1},\underbrace{c_q,...,c_q}_{3n-1},\underbrace{e_r,...,e_r}_{3n-1}) + \sum_{s=3n-r}^{3n-1} (f - e_r) \frac{\partial}{\partial w_s} \mathbf{P}^{ijk}(\underbrace{a_p,...,a_p}_{3n-1},\underbrace{c_q,...,c_q}_{3n-1},\underbrace{e_r,...,e_r}_{3n-1})  \\
   &+ \mathcal{O}(h^2).
   \end{aligned}
\end{equation}

By exploiting the symmetry property of the blossom function, this expansion leads to the estimate
\begin{equation}
    \begin{aligned}
    \bm{\mathcal{J}}_{pqr}^{ijk} =\bm{\mathcal{J}}^{ijk}(a_p, c_q, e_r) + \mathcal{O}(h^2).
    \end{aligned}
\end{equation}

Furthermore, using the global lower bound $\bm{\mathcal{J}} \geq \rho$, we derive
\begin{equation}
\begin{aligned}
    \bm{\mathcal{J}}_{pqr}^{ijk} &= \bm{\mathcal{J}}^{ijk}(a_p, c_q, e_r) +\bm{\mathcal{J}}_{pqr}^{ijk} -\bm{\mathcal{J}}^{ijk}(a_p, c_q, e_r)
    &\geq \bm{\mathcal{J}}(a_p, c_q, e_r)-\gamma h^2 
    &\geq \rho - \gamma h^2,
\end{aligned}
\end{equation}
where $\gamma$ is a constant dependent on the maximum derivative magnitude. Substituting $h = \frac{1}{(3n-1)\sigma}$ yields
\begin{equation} 
    \gamma h^2 = \frac{\gamma}{(3n-1)^2\sigma^2}. 
\end{equation}

Therefore, for sufficiently large subdivision level $\sigma$ such that $\gamma h^2 \ll \rho$, it follows that
\begin{equation} 
   \bm{\mathcal{J}}_{pqr}^{ijk} > 0,
\end{equation}
which implies that all B\'ezier coefficients of the Jacobian determinant in sub-block $\bm{\mathcal{T}}^{ijk}$ share the same sign. This completes the proof.
\end{proof}

\section{Numerical Experiments}
\label{sec3:numerical experiments}

To evaluate the correctness and computational efficiency of the proposed algorithm, a series of numerical experiments were conducted on a MacBook Pro equipped with an Apple M1 Pro processor and 16~GB of RAM. The implementation was based on a hybrid MATLAB/C++ framework: core computational routines were developed in C++ and integrated into MATLAB via MEX interfaces to improve performance, while high-level control flow and visualization were handled within MATLAB.

The experiments focused on the rapid precomputation of the Jacobian determinant coefficients $\bm{\mathcal{J}}_{pqr}$ for three-dimensional B\'ezier volumes. The control points were generated over structured grids, with B\'ezier degrees $n$ ranging from 1 to 10. The algorithm proceeds by first computing directional differences of control points, followed by efficient traversal of all index combinations associated with $\bm{\mathcal{J}}_{pqr}$. This is achieved using parallelized MEX functions and pre-cached binomial coefficients to minimize redundant calculations. All numerical computations were performed in double-precision arithmetic to ensure stability and accuracy throughout the evaluation.

\begin{figure}[htbp]
\centering
\subfigure[]{\includegraphics[width=0.35\linewidth]{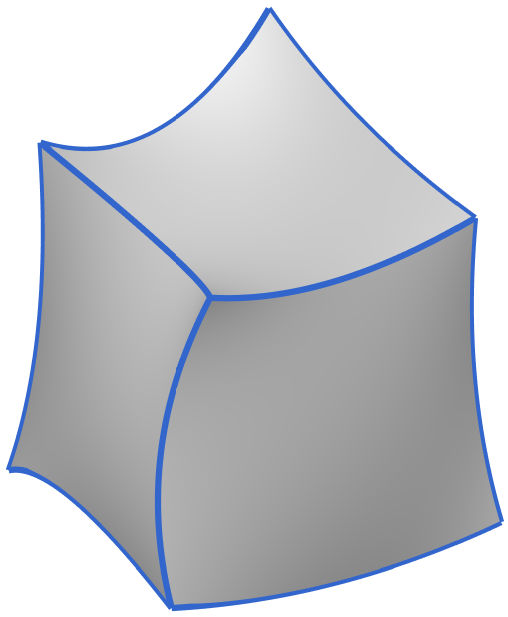}
\label{fig3a:3-2Bezier volume}}
\quad
\subfigure[]{\includegraphics[width=0.45\linewidth]{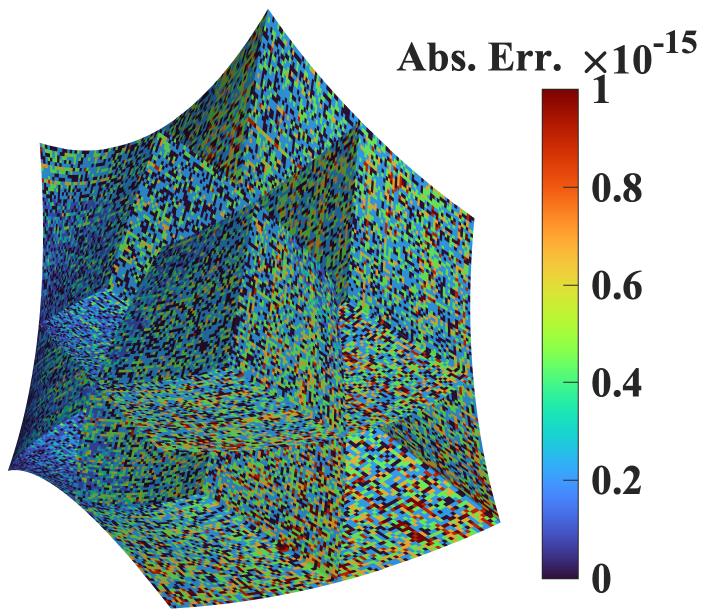}
\label{fig3b:3-2error}}
\\
\subfigure[]{\includegraphics[width=0.35\linewidth]{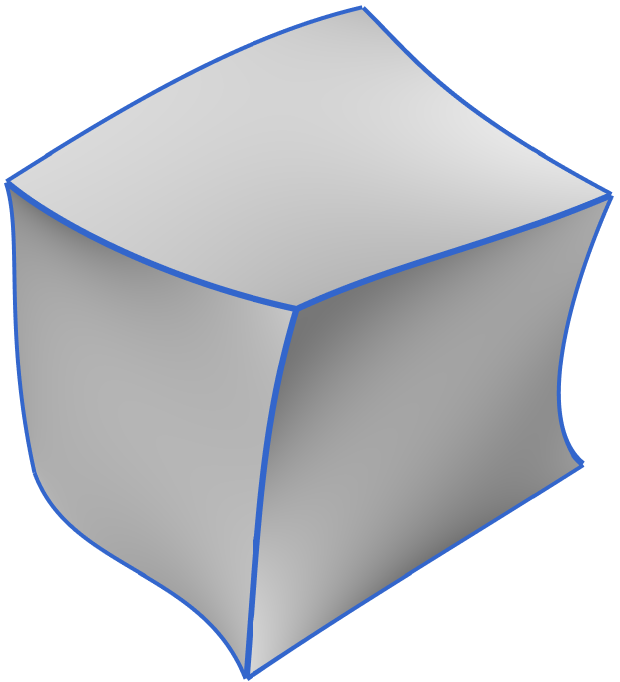}
\label{fig3c:3-3Bezier volume}}
\quad
\subfigure[]{\includegraphics[width=0.45\linewidth]{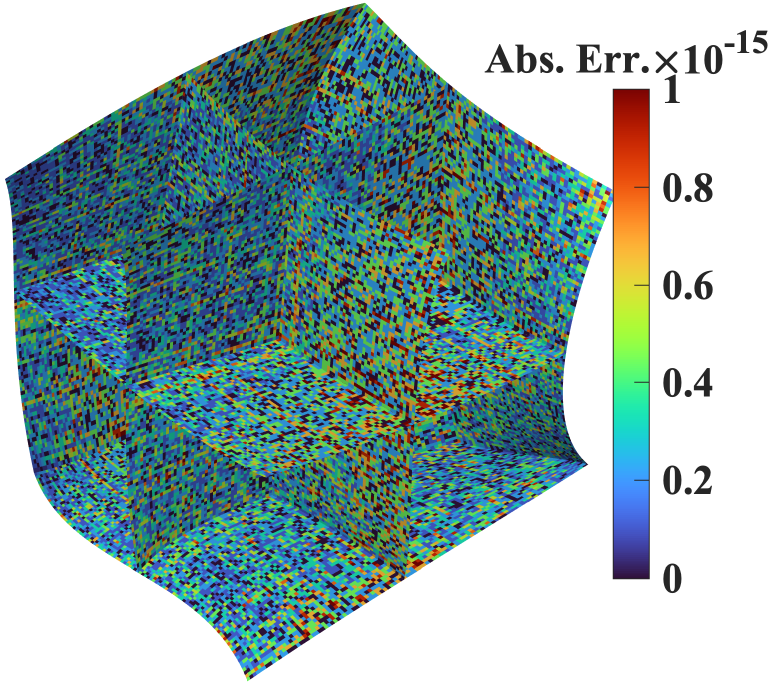}
\label{fig3d:3-3error}}
\caption{Validation of the explicit coefficient formulation of the Jacobian determinant for trivariate B\'ezier volumes.  
(a), (c): Geometric representations of tri-quadratic and tri-cubic B\'ezier volumes, respectively.  
(b), (d): Pointwise absolute error between the analytical and numerically evaluated Jacobian determinants at $101^3$ sample points.}
\label{fig3:bezier_volume_bijectivity}
\end{figure}

\subsection{Validation of the coefficient expressions for the Jacobian determinant}
\label{sec301:validation_jacobian_coefficients}

To validate the correctness of the coefficient expressions for the Jacobian determinant presented in Theorem~\ref{theorem 2}, we computed the coefficients using Eq.~\eqref{eq:Jpqr} and compared them with values obtained through direct numerical evaluation. Specifically, we uniformly sampled 101 points along each parametric direction in the domain $[0,1] \times [0,1] \times [0,1]$. At these sampled points, the Jacobian determinant values were calculated both analytically and numerically, and the absolute errors were used to assess the accuracy.

Figure~\ref{fig3:bezier_volume_bijectivity} illustrates two representative B\'ezier volumes. Figures~\ref{fig3a:3-2Bezier volume} and~\ref{fig3c:3-3Bezier volume} show the geometric configurations of a tri-quadratic and a tri-cubic B\'ezier volume, respectively. The corresponding absolute error distributions in the Jacobian determinant, computed at the sampled points, are presented in Figures~\ref{fig3b:3-2error} and~\ref{fig3d:3-3error}.

The results indicate that the maximum absolute error is on the order of $10^{-15}$, which can be attributed to floating-point round-off errors inherent in numerical computation. It confirms that the Jacobian coefficients derived from Eq.~\eqref{eq:Jpqr} exhibit excellent numerical precision, thereby validating the correctness of the theoretical formulation.

\subsection{Numerical Validation of Computational Efficiency}
\label{sec302:numerical_validation}

In Section~\ref{sec203:Regularity Analysis of Coons Volume Mappings in B\'ezier Representation}, we theoretically analyzed the computational complexity of the coefficients $ \bm{\mathcal{J}}_{pqr} $ of Jacobian determinant involved in Theorem \ref{theorem 2}. To further validate the practical performance of the proposed method, we conducts the numerical experiments to evaluate the algorithm's efficiency and scalability in this section.

\paragraph{Experimental Setup}
In the experiments, we consider 3D B\'ezier mappings on the unit cube with degrees $n = 1, 2, \ldots, 10$, corresponding to $(n+1)^3$ control points. To simulate geometric mappings of general B\'ezier volume, we add random perturbations to each coordinate component of every control points as   
\begin{equation}  
    \bm{P}_{ijk} \leftarrow \bm{P}_{ijk} + \bm{\mathcal{E}},  \quad \bm{\mathcal{E}} \sim \text{Uniform}(0, 0.25)^3,  
\end{equation}  
where $\text{Uniform}(0, 0.25)$ denotes a uniform distribution with a mean of 0 and a variance of 0.25. For each degree, we independently generate $1001$ perturbation samples and compute all the coefficients $ \bm{\mathcal{J}}_{pqr} $ of the Jacobian determinant for each sample. The average computation times and standard deviations are recorded and summarized in Table~\ref{tab:jacobian_time}.


\begin{table}[htbp]
\centering
\caption{Computation time (seconds) for evaluating the coefficients of the Jacobian determinant of B\'ezier volumes with different polynomial degrees. {\color{red}The last column reports the local log--log slope $\bigl(\log T_k-\log T_{k-1}\bigr)/\bigl(\log n_k-\log n_{k-1}\bigr)$, which provides an estimate of the local power-law exponent relating the computation time to the degree $n$. The local slope is not defined for $n=1$ due to the lack of a preceding data point.} Here, \#CtrPts denotes the number of control points of the B\'ezier volume.}
\label{tab:jacobian_time}
\begin{tabular}{cccrc@{\hspace{1em}}|@{\hspace{1em}}cccrc}
    \toprule
    $n$ & \#CtrPts & Mean Time & Std Dev & Local slope &
    $n$ & \#CtrPts & Mean Time & Std Dev & Local slope \\
    \midrule
    1 &  8   & 4.87e-7 & 1.10e-7 &{\color{red}-} & 6 & 343  & 6.70e-2 & 4.46e-3 & {\color{red}7.88}\\
    2 & 27   & 1.99e-5 & 1.48e-6 &{\color{red}5.35} & 7 & 512  & 2.36e-1 & 7.74e-3 & {\color{red}8.16}\\
    3 & 64   & 3.22e-4 & 5.34e-6 &{\color{red}6.87} & 8 & 729  & 7.08e-1 & 1.50e-2 & {\color{red}8.23}\\
    4 & 125  & 2.75e-3 & 6.86e-5 &{\color{red}7.46} & 9 & 1000 & 1.88e-0 & 3.84e-2 & {\color{red}8.30}\\
    5 & 216  & 1.59e-2 & 2.29e-3 &{\color{red}7.87} &10 & 1331 & 4.51e-0 & 6.34e-2 & {\color{red}8.28}\\
    \bottomrule
\end{tabular}
\end{table}

\paragraph{Results and Analysis}

{\color{red}Table~\ref{tab:jacobian_time} shows that the computation time increases steadily with the polynomial degree. The local log--log slopes reported in the last column of the table indicate a strong polynomial dependence of the computation time on the degree $n$, with the slope increasing and stabilizing for higher degrees. When expressed in terms of the number of control points, this trend corresponds to an approximately cubic growth with respect to the control point count (see Fig.~\ref{fig4a:computation_time_with_degree}).} The experimental results confirm the practical operational efficiency and scalability of the algorithm, thus providing a robust basis for future developments and applications. Notably, for low-degree B\'ezier volumes ($n \le 3$), the computation time for Jacobian coefficient is on the order of $10^{-4}$ seconds, meeting the efficiency requirements for engineering and real-time modeling applications.

Figure~\ref{fig4:computational_time} illustrates the trend of computation time versus degree. Figure~\ref{fig4a:computation_time_with_degree} covers the full degree range ($n=1, 2, \ldots, 10$), showing the overall growth pattern of computation time, while Figure~\ref{fig4b:low_degree_time} focuses on the commonly used low-degree cases ($n \le 3$), demonstrating extremely small average computation times and nearly zero standard deviation, highlighting the algorithm's efficiency and numerical stability in practical applications. The experimental results illustrate that the computational method for computing the coefficients of Jacobian determinant proposed in this paper is well-structured and straightforward to implement. {\color{red} Moreover}, it is particularly suitable for the verification of the regularity of higher-order B\'ezier volumes and the fast regularity checking in complex geometric modelling scenarios.

\begin{figure}[H]
    \centering
    \subfigure[Computation time for B\'ezier volumes of degrees $n=1$ to $n=10$]{
    \includegraphics[width=0.45\linewidth]{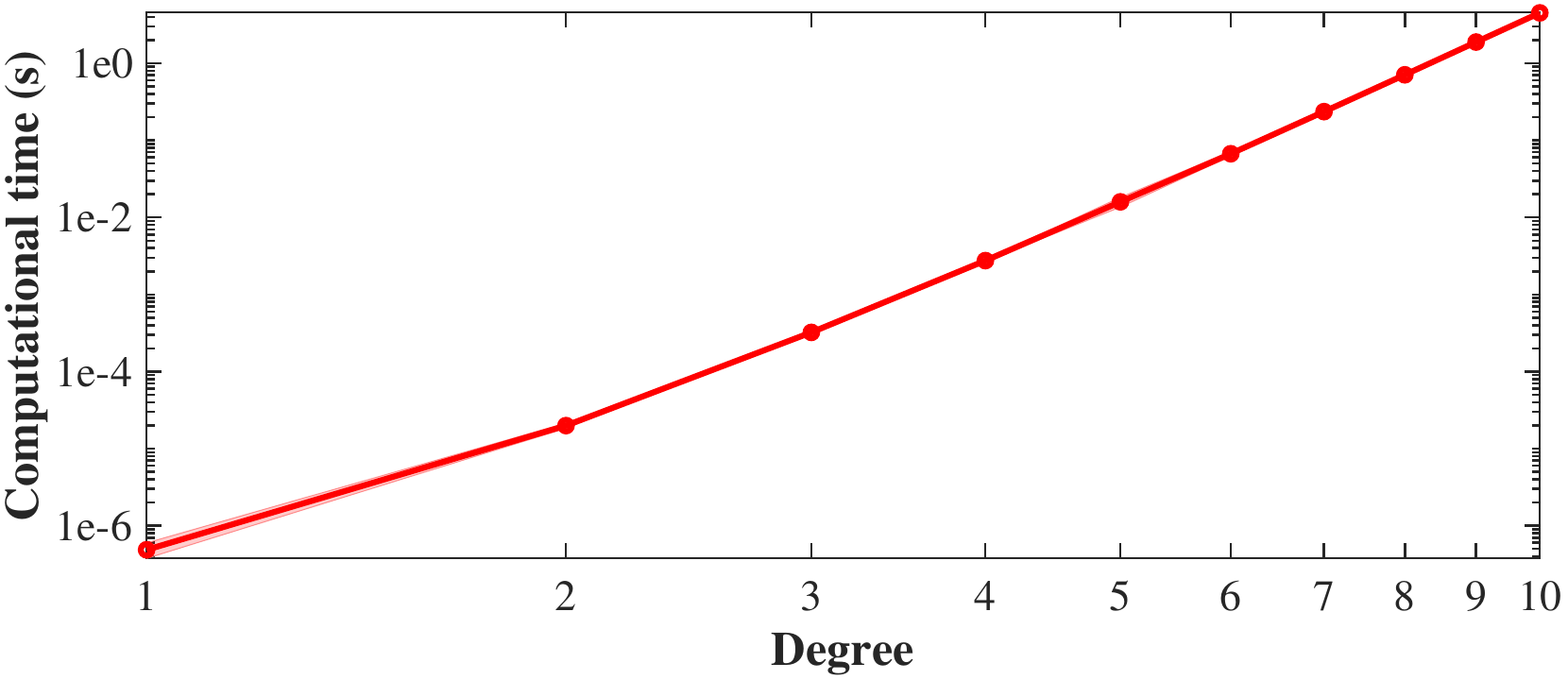}
    \label{fig4a:computation_time_with_degree}
    }
    \quad
    \subfigure[Computation time for low-degree cases ($n=1,2,3$)]{
    \includegraphics[width=0.45\linewidth]{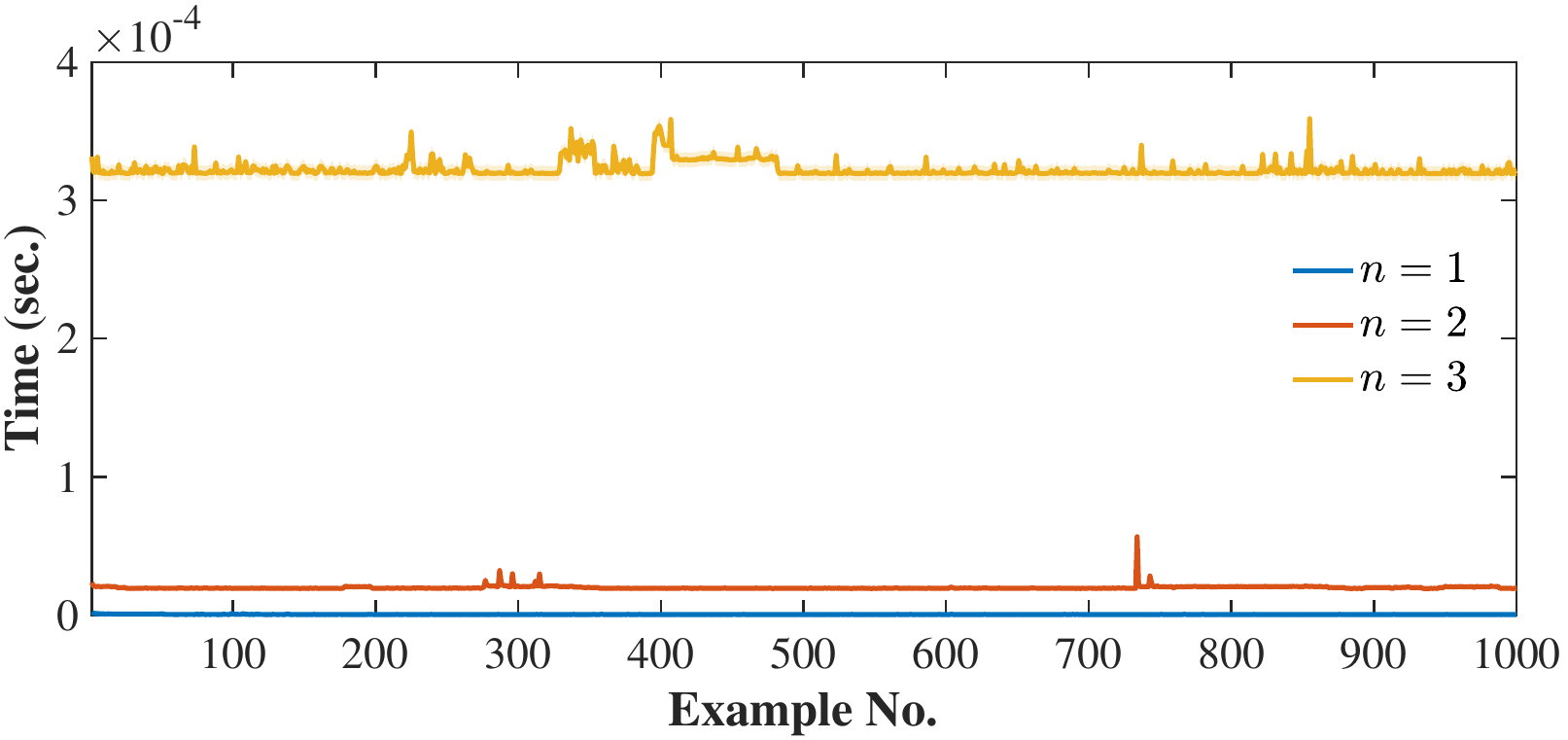}
    \label{fig4b:low_degree_time}
    }
    \caption{Computation time required to evaluate the coefficients of the Jacobian determinant for trivariate B\'ezier volumes with varying degrees.  
    (a) shows the overall trend from degree $n=1$ to $n=10$;  
    (b) provides a comparison of low-degree cases ($n=1,2,3$).}
    \label{fig4:computational_time}
\end{figure}

\subsection{Generalization to regularity verification for arbitrary B-spline volumes}
\label{sec303:Generalization}
Using the B\'ezier extraction technique for B-spline volumes~\cite{borden2011isogeometric}, the proposed regularity verification method can be naturally extended to general B-spline volumes. B\'ezier extraction enables the decomposition of a B-spline volume with complex topology and high-order continuity into local B\'ezier elements, thereby restoring a tensor-product structure that aligns with B\'ezier representation within each element. {\color{red}Note that the conversion from B-spline to Bézier is an exact conversion, not an approximation.} For each extracted B\'ezier element, the explicit coefficient formula of the Jacobian determinant proposed in this paper can be directly applied to check the local regularity condition (i.e., positivity of the Jacobian).

\begin{figure}[!htbp]
    \centering
    {\includegraphics[width=0.4\linewidth]{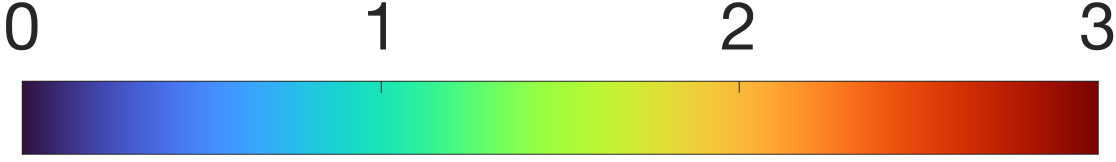}}\\
    \vspace{0.7cm}
    \subfigure[B-spline volume of a blade]{
    \includegraphics[width=0.3\linewidth]{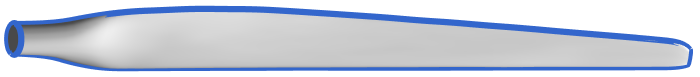}
    \label{fig5a_blade_BSpline}}
    \quad
    \subfigure[Corresponding B\'ezier extraction result (200 B\'ezier elements)]{
    \includegraphics[width=0.3\linewidth]{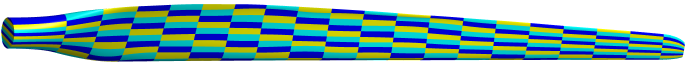}
    \label{fig5b_blade_Bezier}}
    \quad
    \subfigure[Jacobian determinant distribution (time: 3.91 ms)]{
    \includegraphics[width=0.3\linewidth]{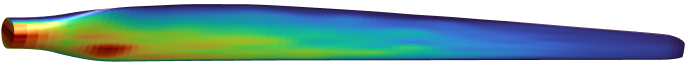}
    \label{fig5c_blade_Jacobian}}\\
    \subfigure[B-spline volume of a propeller]{
    \includegraphics[width=0.3\linewidth]{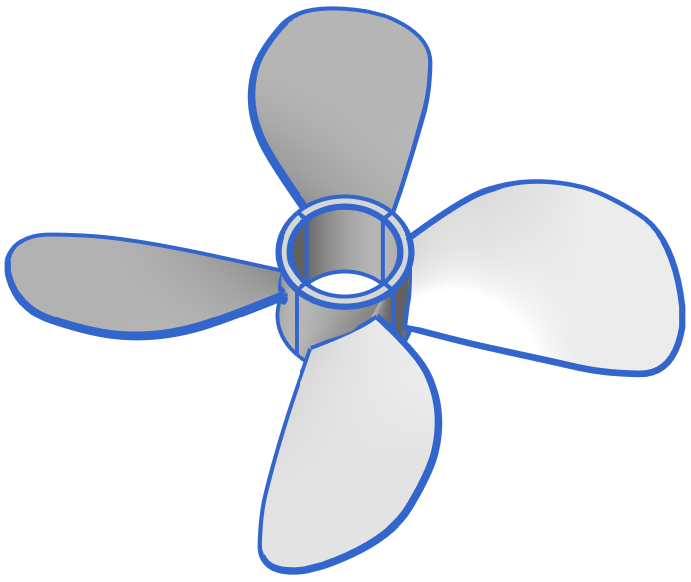}
    \label{fig5d_propeller_BSpline}}
    \quad
    \subfigure[Corresponding B\'ezier extraction result (56 B\'ezier elements)]{
    \includegraphics[width=0.3\linewidth]{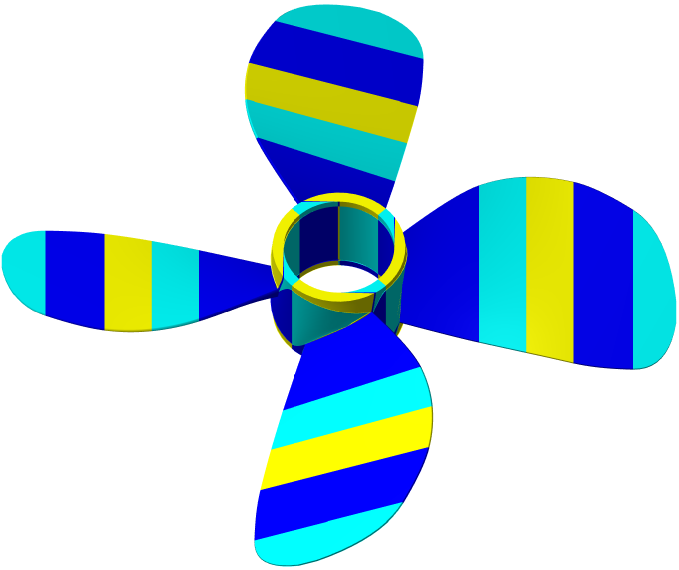}
    \label{fig5e_propeller_Bezier}}
    \quad
    \subfigure[Jacobian determinant distribution (time: 1.14 ms)]{
    \includegraphics[width=0.3\linewidth]{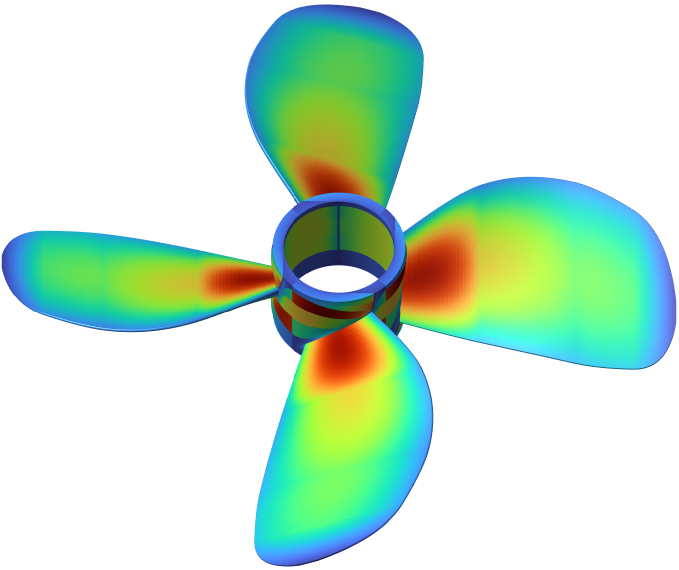}
    \label{fig5f_propeller_Jacobian}}
    \caption{Regularity verification of B-spline volume mappings for two representative geometries.  
    (a)–(c) Blade model: B-spline volume representation, B\'ezier extraction result, and the computed Jacobian determinant distribution.  
    (d)–(f) Propeller model: B-spline volume representation, B\'ezier extraction result, and the computed Jacobian determinant distribution.}
    \label{fig6:B-spline_solids}
\end{figure}

Figure~\ref{fig6:B-spline_solids} demonstrates the regularity verification process for two typical geometric structures in practical engineering applications. The first row shows a blade model constructed from a single B-spline volume, where Figure~\ref{fig5a_blade_BSpline} displays its original B-spline structure. After B\'ezier extraction, 200 B\'ezier elements are obtained, as shown in Figure~\ref{fig5b_blade_Bezier}, and the values of the Jacobian determinant for each element are computed using the proposed algorithm. The distribution of these values is shown in Figure~\ref{fig5c_blade_Jacobian}, with a total computation time of 3.91 milliseconds.

The second row presents a propeller model composed of 8 joined B-spline volumes. Figure~\ref{fig5d_propeller_BSpline} shows its B-spline volume representation, while Figure~\ref{fig5e_propeller_Bezier} displays the 56 extracted B\'ezier elements. The corresponding Jacobian determinant distribution is shown in Figure~\ref{fig5f_propeller_Jacobian}, requiring only 1.14 milliseconds of computation time. This experiment further demonstrates that the proposed method exhibits excellent scalability and efficiency, capable of handling regularity verification requirements for complex structures in practical applications.

In isogeometric analysis, the regularity of geometric mappings is a fundamental requirement for ensuring the convergence and accuracy of numerical methods~\cite{ji2022penalty}. The regularity verification framework proposed in this study exhibits broad applicability to B-spline volumes of arbitrary degree and topological complexity, making it a valuable tool for pre-simulation quality assessment of parametric domain construction.

\section{Conclusions}
\label{sec4:conclusions}

This paper focuses on the problem of regularity verification for parameter mapping of Coons volume mappings, and proposes a determination method based on explicit coefficient expression of Jacobian determinant. The approach is grounded in a systematic theoretical framework, supported by numerical validation and performance analysis. By exploiting the tensor-product structure of Bernstein basis functions inherent to B\'ezier parameterizations, the Jacobian determinant is expressed in a closed form, enabling rapid and robust regularity assessment for B-spline volume mappings.

Building upon the established work in regularity verification for two-dimensional NURBS parameterizations~\cite{ji2021constructing, li2025fast}, this paper extends the analysis to three-dimensional B-spline volumes. Utilizing B\'ezier extraction techniques~\cite{borden2011isogeometric}, arbitrary B-spline volumes are decomposed into local B\'ezier elements, allowing the proposed coefficient-based Jacobian conditions to be applied at the element level. This facilitates global regularity verification over complex volumetric domains. Numerical experiments across a range of polynomial degrees and parameterization configurations confirm the correctness and the computational efficiency of the proposed method. In addition, the proposed framework is particularly well-suited for automated regularity checking and quality assessment of parametric domains in geometric modeling, isogeometric analysis, and integrated CAD/CAE workflows. The associated theory offers a reliable foundation for geometry validation in the simulation of complex engineering structures.

Furthermore, the method proposed in this paper demonstrates excellent extensibility. {\color{red} Similar to the discussion in Section~\ref{sec303:Generalization}, due} to its inherent compatibility with B\'ezier representations, the proposed algorithm can be directly extended to more sophisticated parameterization structures such as T-spline volumes \cite{scott2011isogeometric}. This adaptability stems from the fact that T-splines can likewise be converted into assemblies of B\'ezier elements through extraction techniques, allowing our method to be effectively applied to regularity analysis and pre-simulation mesh quality control in T-spline parametric mappings. {\color{red}Finally, we note that the partitioning of a given object into injectively parametrised volumes represents an interesting and practically relevant direction for future work.}

\section*{Acknowledgements}
\label{sec:Acknowledgement}
This research was supported by the National Natural Science Foundation of China (Nos.~12301490, 12471358), the Youth Project of the Education Department of Liaoning Province (No.~JYTQN2023262), the Youth Talent Innovation Program of Dalian (No.~2023RQ088), and the China Postdoctoral Science Foundation (No.~2025M773071).


\bibliographystyle{elsarticle-num}
\bibliography{references}






\end{document}